\documentclass{birkjour}

\usepackage{amssymb}
\usepackage{cite}

 \newtheorem{theorem}{Theorem}[section]
 \newtheorem{corollary}[theorem]{Corollary}
 \newtheorem{lemma}[theorem]{Lemma}
 \newtheorem{proposition}[theorem]{Proposition}
 \theoremstyle{definition}
 
 \theoremstyle{remark}
 \newtheorem{remark}[theorem]{Remark}
 
 \numberwithin{equation}{section}

\newcommand{\R}{{\mathbb R}}
\newcommand{\N}{{\mathbb N}}
\newcommand{\C}{{\mathbb C}}

\newcommand{\K}{{\mathbb K}}
\newcommand{\M}{{\mathbb M}}
\newcommand{\la}{\langle}
\newcommand{\ra}{\rangle}
\newcommand{\Rp}{{\R_+^*}}

\newcommand{\FC}{\mathcal{FC}_{b}^\infty(\mathcal D(X),\K(X))}
\newcommand{\FCM}{\mathcal{FC}_{b}^\infty(\mathcal D(X),\M(X))}

\newcommand{\FCC}{\mathcal{FC}_{b}^\infty(\mathcal D(\hat X),\Gamma(\hat X))}

\newcommand{\FCK}{\mathcal{FC}_{b}^\infty(\mathcal D(\hat X),\K(X))}

\begin{document}

%
%
%
%
%
%
%
%
%

\title[Laplace operators in gamma analysis]
 {Laplace operators in gamma analysis}

\author[Hagedorn]{Dennis~Hagedorn}

\address{%
Fakult\"at f\"ur Mathematik\\
Universit\"at Bielefeld\\
Postfach 10 01 31\\
D-33501 Bielefeld\\
Germany}

\email{dhagedor@math.uni-bielefeld.de}

\author[Kondratiev]{Yuri~Kondratiev}

\address{%
Fakult\"at f\"ur Mathematik\\
Universit\"at Bielefeld\\
Postfach 10 01 31\\
D-33501 Bielefeld\\
Germany}

\email{kondrat@mathematik.uni-bielefeld.de}

\author[Lytvynov]{Eugene~Lytvynov}

\address{Department of Mathematics\\
Swansea University\\
Singleton Park\\
Swansea\\
SA2 8PP\\
U.K}
\email{e.lytvynov@swansea.ac.uk}

\author[Vershik]{Anatoly~Vershik}

\address{%
Laboratory of Representation Theory
and Computational Mathematics\\
St.Petersburg Department of Steklov Institute of Mathematics\\
27 Fontanka\\
St.Petersburg 19102\\
Russia}

\email{vershik@pdmi.ras.ru}




\date{}

\begin{abstract}
Let $\K(\R^d)$ denote the cone of discrete Radon measures on $\R^d$. The gamma measure $\mathcal G$ is the probability measure on $\K(\R^d)$ which is a measure-valued L\'evy process with intensity measure $s^{-1}e^{-s}\,ds$ on $(0,\infty)$. We study a class of Laplace-type  operators in $L^2(\K(\R^d),\mathcal G)$. These operators are defined as generators of certain (local) Dirichlet forms. The main result of the papers is the essential self-adjointness of these operators on a set of `test' cylinder functions on $\K(\R^d)$.  
\end{abstract}
 
\maketitle
\section{Introduction}\label{drre6eiu}

Handling and modeling complex systems have become an essential part of modern science. For a long time, complex systems have been treated in physics, where e.g.\ methods of probability theory are used to determine their macroscopic behavior by their microscopic properties. Nowadays, complex systems, including ecosystems, biological populations, societies, and financial markets, play an important role in various fields, like biology, chemistry, robotics, computer science, and social science.

A mathematical tool to study complex  systems is infinite dimensional analysis. Such studies
are often related to a probability measure $\mu$ defined on an infinite dimensional state space.
The most `traditional' example of a measure $\mu$ is  Gaussian (white noise) measure, which is defined on the Schwartz space of tempered distributions, $\mathcal S'(\R^d)$, see e.g.\  \cite{BK,B,HKPS}.
Another example of measure $\mu$ is Poisson random measure on $\R^d$. This is a probability measure on the configuration space $\Gamma(\R^d)$ consisting of all locally finite subsets of $\R^d$. A configuration $\gamma=\{x_i\}\in\Gamma(\R^d)$ may be interpreted either as a collection of indistinguishable physical particles located at points $x_i$, or as a population of a species whose individuals occupy points $x_i$, or otherwise depending on the type of the problem.
The Poisson measure corresponds to a system without interaction between its entities. In order to describe an interaction, one introduces Gibbs perturbations of the Poisson measure, i.e., Gibbs measures on $\Gamma(\R^d)$.

In  papers \cite{AKR,AKR2}, some elements of analysis and geometry on the configuration space $\Gamma(\R^d)$ were introduced. In particular, for each $\gamma=\{x_i\}\in\Gamma(\R^d)$, a tangent space to $\Gamma(\R^d)$ at point $\gamma$ was defined as
$$ T_\gamma(\Gamma):= L^2(\R^d\to\R^d,\gamma),$$
where we identified $\gamma$ with the Radon measure $\sum_{i}\delta_{x_i}$. A gradient of a differentiable function $F:\Gamma(\R^d)\to\R$ was explicitly identified as a function
$$\Gamma(\R^d)\ni\gamma\mapsto(\nabla^\Gamma F)(\gamma)\in T_\gamma(\Gamma).$$
This, in turn, led to a  Dirichlet form
$$
\mathcal E^\Gamma(F,G)=\int_{\Gamma(\R^d)}\la (\nabla^\Gamma F)(\gamma),(\nabla^\Gamma G)(\gamma)\ra_{T_\gamma(\Gamma)}\,d\mu(\gamma),
$$
where $\mu$ is either  Poisson measure or a Gibbs measure.
Denote by $-L^\Gamma$ the generator of the Dirichlet form $\mathcal E^\Gamma$. Then, in the case where $\mu$ is Poisson measure, the operator $L^\Gamma$ can be understood as  a Laplace operator on the configuration space $\Gamma(\R^d)$.

Assume that the dimension $d$ of the underlying space $\R^d$ is $\ge2$.  By using the theory of Dirichlet forms, it was shown that there exists a diffusion process on $\Gamma(\R^d)$ which has generator $L^\Gamma$, see \cite{AKR,AKR2,MRpaper,osada,yoshida}. In particular, this diffusion process has  $\mu$ as an invariant measure. (For $d=1$, in order to construct an associated diffusion process an extension of $\Gamma(\R^d)$ is required.)


A further fundamental example of a probability measure on an infinite dimensional space is given by  the gamma measure \cite{TsVY,VGG1,VGG2,VGG3}. This measure, denoted in this paper  by $\mathcal G$,
was initially defined through its Fourier transform as a probability measure on the Schwartz space of tempered distributions, $\mathcal S'(\R^d)$. White noise analysis related to the gamma measure was initiated by Kondratiev, da Silva, Streit, and Us in \cite{KSSU}, and further developed in \cite{KL,L1,L2}. Note that the gamma measure belongs to the class of five Meixner-type L\'evy measures (this class also includes Gaussian and Poisson measures). Each measure $\mu$ from this Meixner-type class admits a `nice' orthogonal decomposition of $L^2(\mu)$ in orthogonal polynomials of infinitely many variables.
In particular, in the case of the gamma measure $\mathcal G$, these orthogonal polynomials are an infinite dimensional counterpart of the  Laguerre polynomials on the real line \cite{KSSU}.

A more delicate analysis shows that the gamma measure
 is  concentrated on the smaller space $\mathbb M(\R^d)$ of all Radon measures on $\R^d$. More precisely, $\mathcal G$
is concentrated on the cone of discrete Radon measures on $\R^d$, denoted by $\K(\R^d)$. By definition, $\K(\R^d)$ consists of all Radon measures of the form $\eta=\sum_i s_i\delta_{x_i}$. It should be stressed that, with $\mathcal G$-probability one, the countable set of positions, $\{x_i\}$, is dense in $\R^d$. As for the weights $s_i$, with $\mathcal G$-probability one, we have $\eta(\R^d)=\sum_i s_i=\infty$, but for each compact set $A\subset\R^d$,
$\eta(A)=\sum_{i:\, x_i\in A}s_i<\infty$. Elements $\eta\in\K(\R^d)$ may model, for example, biological systems, so that the points $x_i$ represent location of some organisms, and the values $s_i$ are a certain attribute attached to these organisms, like their weight or height.

A very important property of the gamma measure is that it is quasi-invariant with respect to a natural group of transformations  of the weights $s_i$ \cite{TsVY}, see also \cite{LS}.
Note also that an infinite dimensional analog of the Lebesgue measure is absolutely continuous with respect to the gamma measure \cite{TsVY,V}.

In   paper \cite{KLV}, which is currently in preparation, we introduce elements of  differential structure on the space of Radon measures, $\mathbb M(\R^d)$.
 More precisely, for a differentiable function $F:\M(\R^d)\to\nolinebreak\R$, we  define its gradient $(\nabla^{\M}F)(\eta)$ as a function of $\eta\in\M(\R^d)$ taking value at $\eta$ in a tangent space $T_\eta(\M)$ to $\M(\R^d)$ at point $\eta$.
Furthermore, we identify  a class of measure-valued L\'evy processes $\mu$ which are probability measures on $\K(\R^d)$ and which admit an integration by parts formula. This class of measures $\mu$ includes the gamma measure $\mathcal G$ as an  important example.
We introduce and study the corresponding
Dirichlet form
$$
\mathcal E^\M(F,G)=\int_{\K(\R^d)}\la(\nabla^\M F)(\eta), (\nabla^\M G)(\eta)\ra_{T_\eta(\M)}\,d\mu(\eta).
$$
In particular,   we find  an explicit form of the generator $-L^\M$ of this Dirichlet form on a proper set of `test' functions on $\K(\R^d)$. Note that the operator $L^\M$ can, in a certain sense, be thought of as a Laplace operator on $\K(\R^d)$, associated with the  measure $\mu$.

In this paper, we will discuss a class of Laplace-type operators associated with the gamma measure $\mathcal G$. More precisely, we will consider a Dirichlet form
$$
\mathcal E^\M(F,G)=\int_{\K(\R^d)}\la (\nabla^\M F)(\eta),c(\eta) (\nabla^\M G)(\eta)\ra_{T_\eta(\M)}\,d\mathcal G(\eta),
$$
where $c(\eta)$ is a certain coefficient (possibly equal identically to one). We prove that this bilinear form is closable, its closure is a Dirichlet form and derive the generator $-L^{\M}$ of  this form. The main result of the paper is that, under some assumption on the coefficient $c(\eta)$, the operator $L^\M$ is essentially self-adjoint on a proper set of `test' functions on $\K(\R^d)$.

Unfortunately, our result does not yet cover the case where $c(\eta)$ is identically equal to one. The open problem here is to prove the essential self-adjointness of a certain differential operator on $\R^d\times(0,\infty)$.

Let us briefly discuss the structure of the paper. In Section~\ref{buyutfughygy}, we recall basic   notions related to differentiation on $\M(\R^d)$, like a tangent space and a gradient of a function on $\M(X)$, see \cite{KLV}. As intuitively clear,  we have two types  of such objects: one related to transformations of the support of a Radon measure,  which we  call intrinsic transformations, and one related to transformations of masses, which we  call extrinsic transformations. We also combine the two types of tangent spaces/gradients into a full tangent space/gradient.

In Section~\ref{gdydytdr},  we explicitly construct the gamma measure $\mathcal G$
on  $\K(\R^d)$.
In Section \ref{jigtt7i9}, we construct and study  the respective Dirichlet forms on the space $L^2(\K(\R^d),\mathcal G)$. These Dirichlet forms are related to the intrinsic, extrinsic, and full gradients. We carry out integration by parts with respect to the measure $\mathcal G$ and derive generators of these bilinear forms.

Finally, in Section \ref{tyde67e}, we prove the essential self-adjointness in $L^2(\K(\R^d),\mathcal G)$ of the generators of the Dirichlet forms on a proper set of `test' functions on $\K(\R^d)$. To this end, we construct a unitary isomorphism between $L^2(\K(\R^d),\mathcal G)$ and the symmetric Fock space $\mathcal F(\mathcal H)$ over the space $$\mathcal H=L^2(\R^d\times(0,\infty),dx\, s^{-1}e^{-s}\,ds)).$$
 We show that the semigroup $(\mathbf T_t)_{t\ge0}$
in $L^2(\K(\R^d),\mathcal G)$ which corresponds to the Dirichlet form
is unitary isomorphic to the second quantization of a respective semigroup
$(T_t)_{t\ge0}$ in  $\mathcal H$. It can be shown that this semigroup   $(T_t)_{t\ge0}$ generates a diffusion on $\R^d\times(0,\infty)$. In particular, in the extrinsic case, the respective diffusion on $\R^d\times(0,\infty)$  is related to a simple space-time transformation of the square of the 0-dimensional Bessel process on $[0,\infty)$.

In the forthcoming paper \cite{HKL},  by using the theory of Dirichlet forms,  we will prove the existence of a diffusion on $\K(\R^d)$ with generator $L^{\M}$. We will also explicitly construct  the Markov semigroup of kernels on $\K(\R^d)$ which corresponds to this diffusion.
Furthermore, we plan to study equilibrium dynamics on $\K(\R^d)$ for which a Gibbs perturbation of the gamma measure (see \cite{HKPR}) is  a symmetrizing (and hence invariant) measure.

\section{Differentiation on the space of Radon measures}\label{buyutfughygy}
In this section, we briefly recall some definitions from \cite{KLV}.

Let $X$ denote the Euclidean space $\R^d$, $d\in\N,$ and  let $\mathcal B(X)$ denote the Borel $\sigma$-algebra on $X$.
Let  $\mathbb M(X)$ denote the space of all (nonnegative) Radon measures on $(X,\mathcal B(X))$. The space $\mathbb M(X)$ is equipped with the vague topology, i.e., the coarsest topology making all mappings
$$\mathbb M(X)\ni\eta\mapsto\langle \varphi,\eta\rangle:= \int_X \varphi\,d\eta,\quad \varphi\in C_0(X),$$
continuous.
Here $C_0(X)$ is the space of all continuous functions on $X$ with compact support. It is well known (see e.g.\ \cite[15.7.7]{Kal}) that $\mathbb M(X)$ is a Polish space.
Let $\mathcal B(\mathbb M(X))$ denote the Borel $\sigma$-algebra on $\mathbb M(X)$.

Let us now introduce an appropriate notion of a gradient $\nabla^{\M}$ of a differentiable function  $F:\M(X)\to\R$.
We start with  transformations of the support, which we call intrinsic transformations. We fix any $v\in C_0^\infty(X\to X)$, a smooth, compactly supported vector field over $X$. Let $(\phi_t^v)_{t\in\R}$ be the corresponding  one-parameter group of diffeomorphisms of $X$  which are equal to the identity outside a compact set in $X$. More precisely, $(\phi_t^v)_{t\in\R}$ is the unique solution of the Cauchy problem
\begin{equation}\label{uigt9}\left\{\begin{aligned}
&\frac{d}{dt}\,\phi_t^v(x)=v(\phi_t^v(x)),\\
&\phi_0^v(x)=x.\end{aligned}
\right.\end{equation}
We naturally lift the action of this group to the space $\M(X)$. For each $\eta\in\M(X)$, we define $\phi_t^v(\eta)\in\M(X)$
as the pushforward of $\eta$ under the mapping $\phi_t^v$. 
Hence, for each $f\in L^1(X,\eta)$,
\begin{equation}\label{fdtrdss}
\la f, \phi_t^v(\eta)\ra=\la f\circ \phi_t^v,\eta\ra.
\end{equation}
For a function $F:\M(X)\to\R$, we define the intrinsic derivative of $F$ in direction $v$ by
\begin{equation}\label{nbufro78r}
(\nabla_v^{\mathrm{int}}F)(\eta):=\frac d{dt}\Big|_{t=0}F( \phi_t^v(\eta)),\quad \eta\in\M(X),
\end{equation}
provided the derivative on the right hand side of formula \eqref{nbufro78r} exists. As an intrinsic tangent space to $\M(X)$ at point $\eta\in \M(X)$ we choose the space $$T_\eta^{\mathrm{int}}(\M):=L^2(X\to X,\eta),$$ i.e., the space of $X$-valued functions
on $X$ which are square integrable with respect to the measure $\eta$.
The intrinsic gradient of $F$ at point $\eta$ is, by definition, the element $(\nabla^{\mathrm{int}}F)(\eta)$ of $T_\eta^{\mathrm{int}}(\M)$
satisfying
\begin{align}(\nabla_v^{\mathrm{int}}F)(\eta)&=((\nabla^{\mathrm{int}}F)(\eta),v)_{T_\eta^{\mathrm{int}}(\M)}\notag\\
&=\int_X \la (\nabla^{\mathrm{int}}F)(\eta,x),v(x)\ra_X\,d\eta(x),\quad v\in C_0^\infty(X\to X).
\label{vufr7u}\end{align}
(In the above formula, $\la\cdot,\cdot\ra_X$ denotes the usual scalar product in $X$.)

We will now introduce transformations of the masses, which we call extrinsic transformations.
 We fix any $h\in C_0(X)$. We consider the one-parameter group of transformations of $\M(X)$ given through multiplication of  each measure $\eta\in\M(X)$ by the function $e^{th(x)}$, $t\in\R$. Thus, for each $\eta\in\M(X)$, we define $M_{th}(\eta)\in\M(X)$ by
\begin{equation}\label{tudr7yer} d M_{th}(\eta)(x):= e^{th(x)}\,d\eta(x).\end{equation}
The extrinsic derivative of a function $F:\M(X)\to\R$ in direction $h$ is defined
by
\begin{equation}\label{hvfy7ed7y}
(\nabla_h^{\mathrm{ext}}F)(\eta):=\frac d{dt}\Big|_{t=0}F( M_{th}(\eta)),\quad\eta\in\M(X),
\end{equation}
provided the derivative on the right hand side of \eqref{hvfy7ed7y} exists.
As an extrinsic tangent space to $\M(X)$ at point $\eta\in \M(X)$ we choose  $$T_\eta^{\mathrm{ext}}(\M):=L^2(X,\eta).$$
The extrinsic gradient of $F$ at point $\eta$ is defined to be the element $(\nabla^{\mathrm{ext}}F)(\eta)$ of $T_\eta^{\mathrm{ext}}(\M)$
satisfying
\begin{align}(\nabla_h^{\mathrm{ext}}F)(\eta)&=((\nabla^{\mathrm{ext}}F)(\eta),h)_{T_\eta^{\mathrm{ext}}(\M)}\notag\\
&=\int_X  (\nabla^{\mathrm{ext}}F)(\eta,x) h(x) \,d\eta(x),\quad h\in C_0(X).
\label{hcdfgUYLF}\end{align}

We finally combine the intrinsic and extrinsic differentiation. For any $\eta\in\M(X)$, the full tangent space to $\M(X)$ at point $\eta$ is defined by
$$T_\eta(\M):=T_\eta^{\mathrm{int}}(\M)\oplus T_\eta^{\mathrm{ext}}(\M).$$ We  define
the full gradient $\nabla^{\M}:=(\nabla^{\mathrm{int}}, \nabla^{\mathrm{ext}})$.

For example, let us consider the set $\FCM$
of all functions $F:\M(X)\to\R$ of the form
\begin{equation}\label{ytf8rtf}F(\eta)=g(\la f_1,\eta\ra,\dots,\la f_N, \eta\ra),\end{equation}
where $g\in C_{b}^\infty(\R^N)$ (an infinitely differentiable function on $\R^N$ which, together with all its derivatives, is bounded), $f_1\,\dots,f_N\in\mathcal D(X)$, and $N\in\N$. Here $\mathcal D(X):=C_0^\infty(X)$ is the space of all smooth, compactly supported functions on $X$. An easy calculation shows that
\begin{align}
(\nabla^{\mathrm{int}}F)(\eta,x)&=\sum_{i=1}^N(\partial_i g)(\la f_1,\eta\ra,\dots,\la f_N,\eta\ra)\nabla f_i(x),\label{vtydr7}\\
(\nabla^{\mathrm{ext}}F)(\eta,x)&=\sum_{i=1}^N(\partial_i g)(\la f_1,\eta\ra, \dots,\la f_N,\eta\ra)f_i(x),\label{vtxsgsydr7}
\end{align}
so that
$$(\nabla^{\M}F)(\eta,x)=\sum_{i=1}^N(\partial_i g)(\la f_1,\eta\ra, \dots,\la f_N,\eta\ra)(\nabla f_i,f_i).
$$
Here $\partial_i g$ denotes the partial derivative of $g$ in the $i$-th variable.

\section{Gamma measure}\label{gdydytdr}
 In this section, following \cite{TsVY,KLV}, we will  recall a construction of the gamma measure.
Recall that we denote by $\K(X)$ the cone of  discrete Radon measures on $X$:
$$\mathbb K(X):=\left\{
\eta=\sum_i s_i\delta_{x_i}\in \mathbb M(X) \mid s_i>0,\, x_i\in X
\right\}.$$
Here, $\delta_{x_i}$ is the Dirac measure with mass at $x_i$, the atoms $x_i$ are assumed to be distinct and their total number is at most countable. By convention, the cone $\mathbb K(X)$ contains the null mass $\eta=0$, which is represented by the sum over the empty set of indices $i$. We denote $\tau(\eta):=\{x_i\}$, i.e., the set on which the measure $\eta$ is concentrated. For $\eta\in\K(X)$ and $x\in\tau(\eta)$, we  denote by $s(x)$ the mass of $\eta$ at  point $x$, i.e., $s(x):=\eta(\{x\})$. Thus, each $\eta\in\K(X)$ can be written in the form $\eta=\sum_{x\in\tau(\eta)}s(x)\delta_x$.

As shown in \cite{HKPR}, $\mathbb K(X)\in\mathcal B(\mathbb M(X))$. We denote by $\mathcal B(\mathbb K(X))$ the trace $\sigma$-algebra of $\mathcal B(\mathbb M(X))$ on $\K(X)$.

\begin{proposition}\label{hufuf}
There exists a unique probability measure $\mathcal G$ on 
$(\K(X),\linebreak \mathcal B(\K(X)))$, called the gamma measure, which has Laplace transform
\begin{equation}\label{jkgyutf}\int_{\K(X)}e^{\la \varphi,\eta \ra}\,d\mathcal G(\eta)=\exp\left[-\int_{X}\log(1-\varphi(x))\,dx\right],\quad \varphi\in C_0(X),\ \varphi<1.\end{equation}
\end{proposition}

We will present a constructive proof of this statement, as it will be used throughout the paper.

\begin{proof}[Proof of Proposition \ref{hufuf}] Denote $\Rp:=(0,\infty)$ and define a metric on $\Rp$ by
$$ d_{\Rp}(s_1,s_2):=\left|\log(s_1)-\log(s_2)\right|,\quad s_1,s_2\in\Rp.$$
Then $\Rp$ becomes a locally compact Polish space, and any set of the form $[a,b]$, with $0<a<b<\infty$, is compact.
We denote $\hat X:=X\times\Rp$ and define the  configuration space over $\hat X$ by
$$\Gamma(\hat X):=\big\{\gamma\subset \hat X\mid |\gamma\cap\Lambda|<\infty\text{ for each compact }\Lambda\subset \hat X\,\big\}.
$$
Here $|\gamma\cap\Lambda|$ denotes the number of points in the set $\gamma\cap\Lambda$.
One can identify a configuration $\gamma\in \Gamma(\hat X)$ with  Radon measure $\sum_{(x,s)\in\gamma}\delta_{(x,s)}$
from $\mathbb M(\hat X)$.
The space $\Gamma(\hat X)$ is endowed with the vague topology, i.e., the weakest topology on $\Gamma(\hat X)$
with respect to which all maps
$$\Gamma(\hat X)\mapsto \langle f,\gamma\rangle:=\int_{\hat X}f(x,s)\,d\gamma(x,s)=\sum_{(x,s)\in\gamma}f(x,s),\quad f\in C_0(\hat X), $$
are continuous.
Let $\mathcal B(\Gamma(\hat X))$ denote the Borel $\sigma$-algebra on $\Gamma(\hat X)$. We denote by $\pi$ the Poisson measure on $(\Gamma(\hat X),\mathcal B(\Gamma(\hat X)))$ with intensity measure
\begin{equation}\label{5}
d\varkappa(x,s):=dx\,d\lambda(s),
\end{equation}
 where
\begin{equation}\label{biulgtfi}d\lambda(s):=\frac{1}{s}\,e^{-s}\,ds.\end{equation}
The measure $\pi$ can be characterized as the unique probability measure on $\Gamma(\hat X)$ which satisfies the Mecke identity: for each measurable function $F:\Gamma(\hat X)\times\hat X\to[0,\infty]$, we have
\begin{multline}\label{Mecke}\int_{\Gamma(\hat X)}d\pi(\gamma)\int_{\hat X}d\gamma(x,s)\,F(\gamma,x,s)\\
=\int _{\Gamma(\hat X)}d\pi(\gamma)\int_{\hat X}d\varkappa(x,s)\,F(\gamma\cup\{(x,s)\},x,s).\end{multline}

Denote by $\Gamma_p(\hat X)$  the set of so-called pinpointing configurations in  $\hat X$. By definition, $\Gamma_p(\hat X)$ consists of all configurations $\gamma\in\Gamma(\hat X)$ such that if $(x_1,s_1),(x_2,s_2)\in\gamma$ and $(x_1,s_1)\ne(x_2,s_2)$, then $x_1\ne x_2$. Thus, a configuration $\gamma\in \Gamma_p(\hat X)$ can not contain two  points $(x,s_1)$ and $(x,s_2)$ with $s_1\ne s_2$. As easily seen, $\Gamma_p(\hat X)\in\mathcal B(\Gamma(\hat X))$. Since the Lebesgue measure $dx$ is non-atomic, the set
$$\big\{(x_1,s_1,x_2,s_2)\in\hat X^2\mid x_1=x_2\big\}$$
is of zero $\varkappa^{\otimes2}$-measure. Denote by $\mathcal B_c(\hat X)$ the set of all Borel measurable sets in $\hat X$ which have compact closure. Fix any $\Lambda\in\mathcal B_0(\hat X)$. Using the distribution of the configuration $\gamma\cap \Lambda$ under $\pi$ (see e.g.\ \cite{Kal}), we conclude that
$$\pi\big(\gamma\in\Gamma(\hat X)\mid \exists (x_1,s_1), (x_2,s_2)\in \gamma\cap\Lambda:\, x_1=x_2,\, s_1\ne s_2\big)=0.$$
Hence, $\pi(\Gamma_p(\hat X))=1$.

For each $\gamma\in\Gamma_p(\hat X)$ and $A\in\mathcal B_c(X)$, we define a local mass by
\begin{equation}\label{iytf86rt}\mathfrak M_A(\gamma):= \int_{\hat X}\chi_A(x)s\,d\gamma(x,s)=\sum_{(x,s)\in\gamma}\chi_A(x)s\in[0,\infty].\end{equation}
Here $\chi_A$ denotes the indicator function of the set $A$.
The set of pinpointing configurations with finite local mass is defined by
$$\Gamma_{pf}(\hat X):=\big\{\gamma\in \Gamma_p(\hat X)\mid \mathfrak M_A(\gamma)<\infty\text{ for each  }A\in\mathcal B_c(X)\big\}.$$
As easily seen, $\Gamma_{pf}(\hat X)\in\mathcal B(\Gamma(\hat X))$ and we denote by $\mathcal B(\Gamma_{pf}(\hat X))$ the trace $\sigma$-algebra of  $\mathcal B(\Gamma(\hat X))$ on $\Gamma_{pf}(\hat X)$. For each $A\in\mathcal B_c(X)$, using the Mecke identity \eqref{Mecke}, we get
$$\int_{\Gamma_p(\hat X)}\mathfrak M_A (\gamma)\,d\pi(\gamma)=\int_{\Gamma_p(\hat X)}d\pi(\gamma)\int_A d\varkappa(x,s)\,s=\int_A dx<\infty.$$
Therefore, $\pi(\Gamma_{pf}(\hat X))=1$ and we can consider $\pi$ as a probability measure on $(\Gamma_{pf}(\hat X),\linebreak[0] \mathcal B(\Gamma_{pf}(\hat X)))$.

We  construct a bijective mapping
$\mathcal R  :\Gamma_{pf}(\hat X)\to\mathbb K(X) $ by setting, for each $\gamma=\{(x_i,s_i)\}\in \Gamma_{pf}(\hat X)$,
$\mathcal R\gamma:=\sum_i s_i\delta_{x_i}\in \mathbb K(X)$.
By \cite[Theorem~6.2]{HKPR}, we have
$$\mathcal B(\mathbb K(X))=\big\{\mathcal R A\mid A\in\mathcal B(\Gamma_{pf}(\hat X))\big\}.$$
Hence, both $\mathcal R$ and its inverse $\mathcal R^{-1}$ are measurable mappings.
We define $\mathcal G$ to be the pushforward of the measure $\pi$ under $\mathcal R$. One can easily check that $\mathcal G$ has Laplace transform \eqref{jkgyutf} and this Laplace transform  uniquely characterizes this measure. 
\end{proof}

\begin{corollary}\label{vjhhfuf}For  each measurable function $F:\K(X)\times X\to[0,\infty]$, we have
\begin{equation}\label{jhigtf8t}
\int_{\K(X)}d\mathcal G(\eta)\int_X d\eta(x)\,F(\eta,x)=\int_{\K(X)}d\mathcal G(\eta)\int_{\hat X}dx\,\,ds\,e^{-s} F(\eta+s\delta_x,x).
\end{equation}
\end{corollary}

\begin{proof}
By the proof of Proposition \ref{hufuf} (in particular, using the Mecke identity), we see that the left hand side of \eqref{jhigtf8t} is equal to
\begin{align*}
&\int_{\Gamma_{pf}(\hat X)}d\pi(\gamma)\int_{\hat X}d\gamma(x,s)\, s F(\mathcal R\gamma,x)\\
&\quad=\int_{\Gamma_{pf}(\hat X)}
d\pi(\gamma)\int_{\hat X}d\varkappa(x,s)\,
sF(\mathcal R(\gamma\cup\{(x,s)\}),x),
\end{align*}
which is equal to the right hand side of \eqref{jhigtf8t}.
\end{proof}

\begin{remark}
In fact, identity \eqref{jhigtf8t} uniquely characterizes the gamma measure $\mathcal G$, i.e., if a probability measure $\mu$ on $\K(X)$ satisfies  identity  \eqref{jhigtf8t} with $\mathcal G$ being replaced by $\mu$, then  $\mu=\mathcal G$. See \cite[Theorem~6.3]{HKPR} for a proof of this statement.
\end{remark}

\begin{remark}\label{ufutfr}
By using either the Laplace transform of the gamma measure (formula \eqref{jkgyutf}) or  formula \eqref{jhigtf8t}, one can easily show that the gamma measure has all moments finite, that is, for each  $A\in\mathcal B_c(X)$ and $n\in\N$, we have
\begin{equation}\label{hjugfiy}\int_{\K(X)}\la\chi_A, \eta\ra^n\,d\mathcal G(\eta)=\int_{\K(X)} \eta(A)^n \,d\mathcal G(\eta)<\infty.\end{equation}
\end{remark}

\section{Dirichlet forms}\label{jigtt7i9}

Having arrived at  notions of both a gradient and a tangent space to $\M(X)$, we would like to construct a corresponding Dirichlet form on the space $L^2(\K(X),\mathcal G)$. This, in turn, should lead us, in future, to a diffusion process  on $\K(X)$. In fact, we will consider different types of Dirichlet forms, corresponding to the intrinsic gradient $\nabla^{\mathrm{int}}$, extrinsic gradient $\nabla^{\mathrm{ext}}$, and the full gradient $\nabla^{\M}$. Furthermore, in the case of the intrinsic gradient (full gradient, respectively), we will  use a coefficient in the Dirichlet form which depends on masses only. The sense of this coefficient will become clear below.

A natural candidate for the domain of these bilinear forms (before the closure) seems to be the set $\FCM$, see \eqref{ytf8rtf}. However, as we learnt in \cite{KLV}, the gamma measure does not allow, on this set,  an integration by parts formula  with respect to intrinsic differentiation. In view of this, we will now introduce an alternative set of test functions on $\K(X)$.

Denote by $\mathcal D(\hat X)$ the space of all infinitely differentiable functions on $\hat X$ which have compact support in $\hat X$. In particular, the support of each $\varphi\in\mathcal D(\hat X)$ is a subset of some set $A\times[a,b]$, where $A\in\mathcal B_c(X)$ and $0<a<b<\infty$. We denote by $\FCC$ the set of all cylinder functions $F:\Gamma(\hat X)\to\R$ of the form
\begin{equation}\label{giyfci}F(\gamma)=g(\la \varphi_1,\gamma\ra,\dots,\la \varphi_N,\gamma\ra),
\quad\gamma\in\Gamma(\hat X),
\end{equation}
where $g\in C_{b}^\infty(\R^N)$, $\varphi_1\,\dots,\varphi_N\in\mathcal D(\hat X)$, and $N\in\N$.
Next, we define
\begin{multline*}\FCK\\
 :=\big\{F:\K(X)\to\R\mid F(\eta)=G(\mathcal R^{-1}\eta)\text{ for some }G\in\FCC\big\}.\end{multline*}
For $\varphi\in\mathcal D(\hat X)$ and $\eta\in\K(X)$, we denote
$$ \langle\!\langle \varphi,\eta\rangle\!\rangle:=\la\varphi,\mathcal R^{-1}\eta\ra=\sum_{x\in\tau(\eta)} \varphi(x, s(x)))=\int_X \frac{\varphi(x,s(x))}{s(x)}\,d\eta(x).$$
Then, each function $F\in\FCK$ has the form
\begin{equation}\label{fty6ed6u45} F(\eta)=g\big(\langle\!\langle \varphi_1,\eta\rangle\!\rangle,\dots,\langle\!\langle \varphi_N,\eta\rangle\!\rangle\big),\quad \eta\in\K(X),\end{equation}
with $g,\varphi_1\,\dots,\varphi_N$ and $N$ as in \eqref{giyfci}.

We note that $\FCC$ is a dense subset of $L^2(\Gamma(\hat X),\zeta)$ for any probability measure $\zeta$ on $\Gamma(\hat X)$.
Hence,  $\FCK$ is a dense subset of $L^2(\K(X),\mu)$ for any probability measure $\mu$ on $\K(X)$,  in particular, $\FCK$ is dense in $L^2(\K(X),\mathcal G)$.

For a function $F$ of the form \eqref{fty6ed6u45},
$v\in C_0^\infty(X\to X)$, $h\in C_0(X)$, and $\eta\in\K(X)$,
we easily calculate:
\begin{align*}
&(\nabla_v^{\mathrm{int}}F)(\eta)\\
&=\sum_{i=1}^N(\partial_i g)\big(\langle\!\langle \varphi_1,\eta\rangle\!\rangle,\dots,\langle\!\langle \varphi_N,\eta\rangle\!\rangle\big)\sum_{x\in\tau(\eta)}\la\nabla_y\big|_{y=x} \varphi_i(y,s(x)),v(x)\ra_X\\
&=\sum_{i=1}^N(\partial_i g)\big(\langle\!\langle \varphi_1,\eta\rangle\!\rangle,\dots,\langle\!\langle \varphi_N,\eta\rangle\!\rangle\big)\\
&\qquad\times \int_X \frac1{s(x)}\,\la \nabla_y\big|_{y=x} \varphi_i(y,s(x)),v(x)\ra_X\,d\eta(x),\\
&(\nabla_h^{\mathrm{ext}}F)(\eta)\\
&=\sum_{i=1}^N(\partial_i g)\big(\langle\!\langle \varphi_1,\eta\rangle\!\rangle,\dots,\langle\!\langle \varphi_N,\eta\rangle\!\rangle\big)
\sum_{x\in\tau(\eta)}
\frac{\partial}{\partial u}\Big|_{u=s(x)}\varphi(x,u)s(x)h(x)\\
&=\sum_{i=1}^N(\partial_i g)\big(\langle\!\langle \varphi_1,\eta\rangle\!\rangle,\dots,\langle\!\langle \varphi_N,\eta\rangle\!\rangle\big)
\int_X
\frac{\partial}{\partial u}\Big|_{u=s(x)}\varphi(x,u)h(x) \,d\eta(x).
\end{align*}
Hence,
\begin{align}
(\nabla^{\mathrm{int}}F)(\eta,x)&=\sum_{i=1}^N(\partial_i g)\big(\langle\!\langle \varphi_1,\eta\rangle\!\rangle,\dots,\langle\!\langle \varphi_N,\eta\rangle\!\rangle\big)\frac1{s(x)}\,\nabla_y\big|_{y=x} \varphi_i(y,s(x)),\label{aaaa}\\
(\nabla^{\mathrm{ext}}F)(\eta,x)&=\sum_{i=1}^N(\partial_i g)\big(\langle\!\langle \varphi_1,\eta\rangle\!\rangle,\dots,\langle\!\langle \varphi_N,\eta\rangle\!\rangle\big)\frac{\partial}{\partial u}\Big|_{u=s(x)}\varphi(x,u).\label{bbb}
\end{align}

Let $F:\K(X)\to\R$, $\eta\in\K(X)$, and $x\in\tau(\eta)$. We define
\begin{align}(\nabla^X_xF)(\eta):=&\nabla_y\big|_{y=x}F(\eta-s(x)\delta_x+s(x)\delta_y),\label{uyfruf}\\ (\nabla^{\Rp}_{x})(\eta):=&\frac{d}{du}\Big|_{u=s(x)}F(\eta-s(x)\delta_x+u\delta_x),\label{fr8r}
\end{align}
provided the derivatives on the right hand side of \eqref{uyfruf} and  \eqref{fr8r} exist.
Here the variable $y$ is from $X$, $\nabla_y$ is the usual gradient on $X$ in the $y$ variable, and the variable $u$ is from $\Rp$. The following simple result is proven in \cite{KLV}.

\begin{lemma}\label{vftydfty7u}
For each $F\in\FCK$, $\eta\in\K(X)$, and $x\in\tau(\eta)$, we have
\begin{align}
(\nabla^{\mathrm{int}}F)(\eta,x)&=\frac1{s(x)}\,(\nabla^X_xF)(\eta),\label{iugti}\\
(\nabla^{\mathrm{ext}}F)(\eta,x)&=(\nabla^\Rp_{x}F)(\eta).\label{yiggfutfr}
\end{align}
\end{lemma}

We fix a measurable function $c:\Rp\to[0,\infty)$ which is locally bounded.
We define symmetric bilinear forms on $L^2(\K(X),\mathcal G)$ by
\begin{align}
&\mathcal E^{\mathrm{int}}(F,G):=\int_{\K(X)}\la (\nabla^{\mathrm{int}} F)(\eta),c(s(\cdot))(\nabla^{\mathrm{int}} G)(\eta)\ra_{T_\eta^{\mathrm{int}}(\M)}\,d\mathcal G(\eta),\notag\\
&\quad= \int_{\K(X)}d\mathcal G(\eta)\,\int_{ X}d\eta(x)\,\big\la (\nabla^{\mathrm{int}} F)(\eta,x), \, c(s(x))(\nabla^{\mathrm{int}} G)(\eta,x)\ra_X\,,\label{uftuu7r}\\
&\mathcal E^{\mathrm{ext}}(F,G):=\int_{\K(X)}\la (\nabla^{\mathrm{ext}} F)(\eta),(\nabla^{\mathrm{ext}} G)(\eta)\ra_{T_\eta^{\mathrm{ext}}(\M)}\,d\mathcal G(\eta),\label{nbkitr86tr}\\
&\mathcal E^\M(F,G):=\mathcal E^{\mathrm{int}}(F,G)+\mathcal E^{\mathrm{ext}}(F,G),\label{ftfttfftu}
\end{align}
where $F,G\in \FCK$.
It follows from formulas \eqref{aaaa} and   \eqref{bbb} that, for each $F\in \FCK$, there exist a constant $C_1>0$,   a set $A\in\mathcal B_c(X)$
and an interval $[a,b]$ with $0<a<b<\infty$
such that
 \begin{multline}\label{lkgtiuotg}\max\{\|\nabla^{\mathrm{int}}F(\eta,x)\|_X,|\nabla^{\mathrm{ext}}F(\eta,x)|\}\le C_1\,\chi_A(x)\,\chi_{[a,b]}(s(x)),\\ \eta\in\K(X),\ x\in\tau(\eta).\end{multline}
 Since the function $c$ is locally bounded, there exists a constant $C_2>0$ such that
\begin{equation}\label{yulful7fr}
c(s(x)) \chi_{[a,b]}(s(x))\le C_2,\quad \eta\in\K(X),\ x\in\tau(\eta).
\end{equation}
Therefore, by \eqref{hjugfiy}, \eqref{lkgtiuotg},   and \eqref{yulful7fr},   the  integrals in \eqref{uftuu7r} and \eqref{nbkitr86tr} indeed make sense and are finite for any $F,G\in\FCK$.

Using Lemma~\ref{vftydfty7u}, we may also give an equivalent representation of the bilinear forms $\mathcal E^{\mathrm{int}}$, $\mathcal E^{\mathrm{ext}}$.

\begin{lemma}\label{gufu7} For any $F,G\in \FCK$,
\begin{align}
&\mathcal E^{\mathrm{int}}(F,G)=\int_{\K(X)}d\mathcal G(\eta)\int_{\hat X}dx\,ds\,e^{-s}\,\frac{c(s)}{s^2}
\big\la\nabla_x
F(\eta+s\delta_x),\nabla_x
G(\eta+s\delta_x)\big\ra_X\,,\label{ytdey6}\\
&\mathcal E^{\mathrm{ext}}(F,G)=\int_{\K(X)}d\mathcal G(\eta)\int_{\hat X}dx\,ds\,e^{-s}\left(\frac{d}{ds}
\,F(\eta+s\delta_x)\right)\left(\frac{d}{ds}
G(\eta+s\delta_x)\right).\label{itftr68y}
\end{align}
\end{lemma}

\begin{proof} Formulas \eqref{ytdey6}, \eqref{itftr68y} directly follow  from Corollary \ref{vjhhfuf}, Lemma~\ref{vftydfty7u}, and formulas \eqref{uftuu7r}, \eqref{nbkitr86tr}.
\end{proof}

The lemma below shows that the introduced symmetric bilinear forms are well defined on $L^2(\K(X),\mathcal G)$.

\begin{lemma}\label{oACgI} Let $F,G\in \FCK$ and let $F=0$ $\mathcal G$-a.e. Then $\mathcal E^\sharp(F,G)
=0$, $\sharp=\mathrm{int}, \mathrm{ext}, \M$.
\end{lemma}

\begin{proof}
For each $A\in\mathcal B_c(X)$, making use of  Corollary \ref{vjhhfuf}, we get
$$\int_{\K(X)}d\mathcal G(\eta)\int_{\hat X}dx\,ds\,e^{-s} |F(\eta+s\delta_x)|\chi_A(x)=
\int_{\K(X)}d\mathcal G(\eta)\,|F(\eta)|\,\eta(A)=0.$$
Hence $F(\eta+s\delta_x)=0$ $d\mathcal G(\eta)\,dx\,ds$-a.e. on $\K(X)\times\hat X$. From here and Lemma~\ref{gufu7}, the statement  easily follows.
\end{proof}

\begin{lemma}\label{kjhoiy}
For $\sharp=\mathrm{int}, \mathrm{ext}, \M$, the bilinear form $(\mathcal E^\sharp,\FCK)$
 is a pre-Dirichlet form on $L^2(\K(X),\mathcal G)$ (i.e., if it  is closable, then its closure  is a Dirichlet form).
\end{lemma}

\begin{proof}
The assertion follows, by standard methods, directly from \cite[Chap.~I, Proposition 4.10]{MR} (see also \cite[Chap.~II, Exercise 2.7]{MR}).
\end{proof}

Analogously to \eqref{uyfruf}, \eqref{fr8r}, we define, for a function $F:\K(X)\to\R$, $\eta\in\K(X)$, and $x\in\tau(\eta)$,
\begin{align}
&(\Delta^X_{x}F)(\eta):= \Delta_{y}
\big|_{y=x} F(\eta- s(x)\delta_x+s(x)\delta_y),\label{kZFU}\\
	&(\Delta^\Rp_{x}F)(\eta):=\left(\frac{d^2}{du^2}-\frac{d}{du}\right)\Big|_{u=s(x)} F(\eta- s(x)\delta_x+u\delta_x).\label{giyfyuuy}
\end{align}
Here and below, $\Delta$ denotes the usual Laplacian on $X$ ($\Delta_y$ denoting the Laplacian in the $y$ variable).
Explicitly, for  a function $F\in\FCK$ of the form  \eqref{fty6ed6u45}, we get
\begin{align}
&(\Delta^X_{x}F)(\eta)=\sum_{i,j=1}^N (\partial_i\partial_jg)(\la\!\la \varphi_1,\eta\ra\!\ra,\dots,\la\!\la \varphi_N,\eta\ra\!\ra)\notag\\
&\quad \times \la \nabla_y\big|_{y=x} \varphi_i(y,s(\eta, x)),\nabla_y\big|_{y=x} \varphi_j(y,s(x))\ra_X\notag\\
&\quad+
\sum_{i=1}^N (\partial_i g)(\la\!\la \varphi_1,\eta\ra\!\ra,\dots,\la\!\la \varphi_N,\eta\ra\!\ra)\Delta_y\big|_{y=x} \varphi_i(y,s(x)),\label{jfutf}
\end{align}
and similarly, we calculate $(\Delta^\Rp_{x}F)(\eta)$.

\begin{proposition}\label{hiftr8iltrl}  For each $F\in\FCK$, we define
\begin{align}
&(L^{\mathrm{int}}F)(\eta):=\int_X d\eta(x)\,\frac{c(s(x))}{s(x)^2}\,(\Delta^X_x F)(\eta),\label{vfyufdy}\\
&(L^{\mathrm{ext}}F)(\eta):=\int_X d\eta(x)\,(\Delta^\Rp_{x} F)(\eta),\quad \eta\in\K(X),\label{bfuyf}\\
&L^{\M}F:=L_1^{\mathrm{int}}F+L_1^{\mathrm{ext}}F.\label{tye65e}
\end{align}
Then, for $\sharp=\mathrm{int}, \mathrm{ext}, \M$, $(L^{\sharp},\FCK)$ is a symmetric operator in $L^2(\K(X),\mathcal G)$ which satisfies
$$\mathcal E^\sharp(F,G)=(-L^\sharp F,G)_{L^2(\K(X),\mathcal G)}\,,\quad F,G\in\FCK.$$
The bilinear  form $(\mathcal E^\sharp, \FCK)$ is closable on $L^2(\K(X),\mathcal G)$ and its closure, denoted by  $(\mathcal E^\sharp,D(\mathcal E^\sharp))$, is a Dirichlet form. The operator $$(-L^\sharp,\FCK)$$
has Friedrichs' extension, which we denote by $(-L^\sharp, D(L^\sharp))$.
\end{proposition}

\begin{proof} We first note that, for a fixed $F\in\FCK$, there exist $A\in\mathcal B_c(X)$ and an interval $[a,b]$ with $0<a<b<\infty$ such that the functions
$$\hat X\ni(x,s)\mapsto \nabla_x F(\eta+s\delta_x),\quad  \hat X\ni(x,s)\mapsto \frac{d}{ds}\, F(\eta+s\delta_x)$$
vanish outside the set $A\times[a,b]$.
Let $\sharp=\mathrm{int}$ and let $F,G\in\FCK$. Using  Lemma~\ref{gufu7} and integrating by parts in the $x$ variable, we
get
\begin{equation}\label{kjg]t8i}
\mathcal E^{\mathrm{int}}(F,G)=\int_{\K(X)}d\mathcal G(\eta)\int_{\hat X}dx\,ds\,e^{-s}\,\frac{c(s)}{s^2}
\big(-\Delta_x
F(\eta+s\delta_x)\big)
G(\eta+s\delta_x).
\end{equation}
Note that, for $F$ of the form \eqref{fty6ed6u45}, we have
\begin{align}
&(\Delta_{x}F)(\eta+s\delta_x)\notag\\
&=\sum_{i,j=1}^N (\partial_i\partial_jg)\big(\la\!\la \varphi_1,\eta\ra\!\ra+\varphi_1(x,s),\dots,\la\!\la \varphi_N,\eta\ra\!\ra+\varphi_N(x,s)\big)\\
&\quad\times \la \nabla_x \varphi_i(x,s),\nabla_x \varphi_j(x,s)\ra_X\notag\\
&\quad+
\sum_{i=1}^N (\partial_ig)\big(\la\!\la \varphi_1,\eta\ra\!\ra+\varphi_1(x,s),\dots,\la\!\la \varphi_N,\eta\ra\!\ra+\varphi_N(x,s)\big)\Delta_x \varphi_i(x,s).\label{buyfrt67ur}
\end{align}
 Hence,  the function under the sign of integral on the right hand side of  \eqref{kjg]t8i} is integrable.
By Corollary~\ref{vjhhfuf},  \eqref{jfutf}, \eqref{vfyufdy},  \eqref{kjg]t8i}, and \eqref{buyfrt67ur},
we get
\begin{align}
\mathcal E^{\mathrm{int}}(F,G)&=\int_{\K(X)}d\mathcal G(\eta)\int_{\hat X}d\eta(x)\,\frac{c(s(x))}{s(x)^2}\,(-\Delta^X_x F)(\eta,x)G(\eta)\notag\\
&=\int_{\K(X)} (-L^{\mathrm{int}}F)(\eta)G(\eta)\,d\mathcal G(\eta).\label{rdee54w45}
\end{align}
 By \eqref{jfutf} and the local boundedness of the function $c$, there exist $C_3>0$ and $A\in\mathcal B_0(X)$ such that
$$\frac{c(s(x))}{s(x)^2}\,|(\Delta_x F)(\eta)|\le C_3 \chi_A(x),\quad \eta\in\K(X),\ x\in\tau(\eta).$$
Hence, by \eqref{hjugfiy} and \eqref{vfyufdy}, we get  $L^{\mathrm{int}}F\in L^2(\K(X),\mathcal G)$. Thus,  the bilinear form  $(\mathcal E^{\mathrm{int}},\FCK)$ has $L^2$-generator. Hence, the statement of the proposition regarding $\sharp=\mathrm{int}$ holds.

    The proof for $\sharp=\mathrm{ext}$ (and so also for $\sharp=\M$) is similar.
\end{proof}

\begin{remark}\label{giyur7675}
Let us quickly note some natural choices of the coefficient function $c(s)$. Choosing $c(s)=1$,  the intrinsic Dirichlet form becomes the closure of the bilinear form
$$ \mathcal E^{\mathrm{int}}(F,G):=\int_{\K(X)}\la (\nabla^{\mathrm{int}} F)(\eta),(\nabla^{\mathrm{int}} G)(\eta)\ra_{T_\eta^{\mathrm{int}}(\K)}\,d\mathcal G(\eta). $$
The choice of $c(s)=s$ yields, in fact, the  Dirichlet form
 which is associated with a diffusion process on $\K(X)$ of the type $\eta(t)=\sum_{i=1}^\infty s_i\delta_{x_i(t)}$, where $(x_i(t))_{i=1}^\infty$ are independent Brownian motions on $X$, see \cite{HKL}.
When we choose $c(s)=s^2$, the generator of the intrinsic Dirichlet form becomes (see \eqref{vfyufdy})
$$(L^{\mathrm{int}}F)(\eta)=\int_X d\eta(x)\,(\Delta^X_x F)(\eta).$$
\end{remark}

Below we denote by $\FC$ the set of the functions on $\K(X)$ which are restrictions of functions from $\FCM$ to $\K(X)$, i.e., they have the form \eqref{ytf8rtf} with $\eta\in\K(X)$.
We  note that $\FC$ is a dense subset of $L^2(\K(X),\mu)$ for any probability measure $\mu$ on $\K(X)$ (see \cite[Corollary~6.2.8]{Hagedorn} for a proof of this rather obvious statement). In particular, $\FC$ is dense in $L^2(\K(X),\mathcal G)$.
We finish this section with the following proposition.

\begin{proposition}\label{hyufgtf} Assume that the function $c$ satisfies
\begin{equation}\label{uigti8tg}
\int_{\Rp} c(s)e^{-s}\,ds<\infty.
\end{equation}
For $\sharp=\mathrm{int},\mathrm{ext},\M$, we have
\begin{equation}\label{gfufuur6}
 \FC\subset D(\mathcal E^\sharp), \end{equation}
and for any $F,G\in\FC$, $\mathcal E^\sharp (F,G)$ is given by the respective formula in \eqref{uftuu7r}--\eqref{ftfttfftu}.

\end{proposition}

\begin{proof} For $F\in D(\mathcal E^\sharp)$, denote $\mathcal E^\sharp(F):=\mathcal E^\sharp(F,F)$.
On $D(\mathcal E^\sharp)$ we consider the norm
\begin{equation}\label{utyyt}\|F\|_{D(\mathcal E^\sharp)}:=\mathcal E^\sharp(F)^{1/2}+\|F\|_{L^2(\K(X),\mathcal G)}\,.\end{equation}
Let $F\in \FC$, and for simplicity of notation, assume that $F$ is of the form
$F(\eta)=g(\la f,\eta\ra)$, where $g\in C_b^\infty(\R)$ and $f\in\mathcal D(X)$.
For each $n\in\N$, we fix any function $u_n\in C^\infty(\R)$ such that
\begin{equation}\label{ugitgeesa86r} \chi_{[1/n,\,\infty)}\le u_n\le \chi_{[1/(2n),\,\infty)}\end{equation}
and
\begin{equation}\label{gyuyyt}
|u_n'(t)|\le 4n\,\chi_{[1/(2n),\,1/n]}(t),\quad t\in\R.\end{equation}
For $n\in\N$,  let $v_n\in C^\infty(\R)$ be such that
\begin{equation}\label{ucsgsgi}
\chi_{(-\infty,n+1]}\le v_n\le \chi_{(-\infty,n+2]}\end{equation}
 and
\begin{equation}\label{ffuyf}
|v_n'(t)|\le 2\,\chi_{[n+1,\,n+2]}(t),\quad t\in\R.\end{equation}
We define
\begin{equation}\label{gygugyu}
 h_n(s):=s u_n(s)v_n(s),\quad s\in\Rp,\ n\in\N,\end{equation}
and
\begin{equation}\label{rtdrddr}
 \varphi_n(x,s):=f(x)h_n(s),\quad (x,s)\in\hat X,\ n\in\N.\end{equation}
Note that $h_n\in C_0^\infty(\Rp)$ and $\varphi_n\in\mathcal D(\hat X)$.
Let
\begin{equation}\label{e6}
 F_n(\eta):=g(\la\!\la \varphi_n,\eta\ra\!\ra),\quad \eta\in\K(X),\ n\in\N,\end{equation}
each $F_n$ being an element of $\FCK$.
For each $\eta\in\K(X)$,
\begin{equation}\label{dredroiu} \la\!\la \varphi_n,\eta\ra\!\ra=\sum_{x\in\tau(\eta)} f(x)s(x)u_n(s(x))v_n(s(x))\to\la f,\eta\ra\quad \text{as }n\to\infty.\end{equation}
Hence, by the dominated convergence theorem, $F_n\to F$ in $L^2(\K(X),\mathcal G)$.
Note that
\begin{equation}\label{sews5t} F_n(\eta+s\delta_x)=g(\la\!\la \varphi_n,\eta\ra\!\ra+\varphi_n(x,s)),\quad \eta\in\K(X),\ (x,s)\in\hat X.\end{equation}
Using  Lemma~\ref{gufu7} and formulas \eqref{ugitgeesa86r}--\eqref{sews5t}, one can easily show that
\begin{equation}\label{rtew4w}\mathcal E^\sharp (F_n-F_m)\to0\quad\text{as }n,m\to\infty. \end{equation}
Since $(\mathcal E^\sharp,D(\mathcal E^\sharp))$ is a closed bilinear form on $L^2(\K(X),\mathcal G)$, we therefore have $F\in D(\mathcal E^\sharp)$, and furthermore $\mathcal E^\sharp(F_n)\to \mathcal E^\sharp(F)$ as $n\to\infty$. From here, analogously to the proof of \eqref{rtew4w}, we conclude that $\mathcal E^\sharp(F)$ is given by the respective formula in \eqref{uftuu7r}--\eqref{ftfttfftu} with $G=F$.

The statement of the proposition about $\mathcal E^\sharp(F,G)$ for general $F,G\in \FC$ follows from the above statement about $\mathcal E^\sharp(F)$ and the polarization identity.
\end{proof}

\begin{remark}\label{yu7tir6tfr}
Let  $\sharp=\mathrm{int},\mathrm{ext},\M$. For $\sharp=\mathrm{int},\M$, assume that condition \eqref{uigti8tg} is satisfied and the dimension $d$ of the underlying space $X$ is $\ge2$. In the forthcoming paper \cite{HKL}, for $\sharp=\mathrm{int},\mathrm{ext},\M$, we will prove the existence of a conservative diffusion process on $\K(X)$ (i.e., a conservative strong Markov process with continuous sample paths in $\K(X)$) which is properly associated with the Dirichlet form $(\mathcal E^\sharp,D(\mathcal E^\sharp))$, see \cite{MR} for details on diffusion processes properly associated with a Dirichlet form. In particular, this diffusion process is $\mathcal G$-symmetric and has $\mathcal G$ as an invariant measure.
\end{remark}

\begin{remark}\label{jgyuuygfy} Let  $\sharp=\mathrm{int},\mathrm{ext},\M$. Consider the Dirichlet form
$(\mathcal E^\sharp,D_1(\mathcal E^\sharp))$ which is defined as the closure of the bilinear form $(\mathcal E^\sharp,\FC)$.
By Proposition~\ref{hyufgtf}, the Dirichlet form $(\mathcal E^\sharp,D(\mathcal E^\sharp))$ is an extension of the Dirichlet form $(\mathcal E^\sharp,D_1(\mathcal E^\sharp))$, i.e., $D_1(\mathcal E^\sharp)\subset D(\mathcal E^\sharp)$.
So, there is a natural question
 whether these  Dirichlet forms coincide, i.e.,  $D_1(\mathcal E^\sharp)= D(\mathcal E^\sharp)$, or, equivalently,
 whether the
set $\FC$ is dense in in the space $D(\mathcal E^\sharp)$ equipped with norm \eqref{utyyt}.
We do not expect a positive answer to this question.
Furthermore, we do not expect the existence of a conservative diffusion process on $\K(X)$ which is properly associated with the Dirichlet form  $(\mathcal E^\sharp,D_1(\mathcal E^\sharp))$.
\end{remark}

\section{Essential self-adjointness of the generators}\label{tyde67e}

In this section, for $\sharp=\mathrm{int},\mathrm{ext},\M$, we will discuss the essential self-adjointness of the operator $(L^\sharp,D(L^\sharp))$ on the domain $\FCK$.

\begin{theorem}\label{ydred6e6r} Let  $\sharp=\mathrm{int},\mathrm{ext},\M$.
Let the function $c:\Rp\to[0,\infty)$ be measurable  and locally bounded.
For  $\sharp=\M$, assume additionally that
\begin{equation}\label{gyurt68otgyft}
c(s)=a_1 s+ a_2 s^2+a_3s^3
\end{equation}
for some $a_i\ge 0$, $i=1,2,3$, $\max\{a_1,a_2,a_3\}>0$.
Then the operator $(L^\sharp,\linebreak \FCK)$ is essentially self-adjoint on $L^2(\K(X),\mathcal G)$.
\end{theorem}

\begin{proof}Fix any $F\in\FCC$ and $\gamma\in\Gamma(\hat X)$. Consider the function
$$\hat X\setminus\gamma\ni(x,s)\mapsto F(\gamma+\delta_{(x,s)}). $$
It is evident that this function admits a unique extension by continuity to the whole space $\hat X$. We denote the resulting function by $F(\gamma+\delta_{(x,s)})$, although $\gamma+\delta_{(x,s)}$ is not necessarily an element of $\Gamma(\hat X)$. Note that $F(\gamma+\delta_{(x,s)})$ is a smooth functions of $(x,s)\in\hat X$.

We preserve the notation $(\mathcal E^\sharp,D(\mathcal E^\sharp))$ for the realization of the respective Dirichlet form on $\Gamma_{pf}(\hat X)$. Thus, $(\mathcal E^\sharp,D(\mathcal E^\sharp))$ is the closure of the bilinear form
$$(\mathcal E^\sharp,\FCC)$$ on $L^2(\Gamma(\hat X),\pi).$
Furthermore, by the counterpart of Lemma~\ref{gufu7}
for the domain\linebreak $\FCK$, we get, for any $F,G\in\FCC$,
\begin{align}
&\mathcal E^{\mathrm{int}}(F,G)\notag\\
&\quad=\int_{\Gamma_{pf}(\hat X)}d\pi(\gamma)\int_{\hat X}dx\,ds\,e^{-s}\,\frac{c(s)}{s^2}
\big\la\nabla_x
F(\gamma+\delta_{(x,s)}),\nabla_x
G(\gamma+\delta_{(x,s)})\big\ra_X\,,\notag\\
&\mathcal E^{\mathrm{ext}}(F,G)\notag\\
&\quad =\int_{\Gamma_{pf}(\hat X)}d\pi(\gamma)\int_{\hat X}dx\,ds\,e^{-s}\left(\frac{d}{ds}
\,F(\gamma+\delta_{(x,s)})\right)\left(\frac{d}{ds}
G(\gamma+\delta_{(x,s)})\right),\notag\\
&\mathcal E^{\K}(F,G)=\mathcal E^{\mathrm{int}}(F,G)+\mathcal E^{\mathrm{ext}}(F,G).\label{ufr8}
\end{align}

We keep the notation $(L^\sharp,D(L^\sharp))$ for the generator of the closed bilinear form $(\mathcal E^\sharp,D(\mathcal E^\sharp))$ on $L^2(\Gamma_{pf},\pi)$.
 We easily conclude from Proposition~\ref{hiftr8iltrl} that
 $$\FCC\subset D(L^\sharp)$$ and for each $F\in\FCC$ and $\gamma\in\Gamma(\hat X)$
\begin{align}
&(L^{\mathrm{int}}F)(\gamma)=\int_{\hat X}d\gamma(x,s)\,\frac{c(s)}{s}\,(\Delta^X_x F)(\gamma),\label{2}\\
&(L^{\mathrm{ext}}F)(\gamma)=\int_{\hat X}d\gamma(x,s)\,s\,(\Delta^\Rp_x F)(\gamma),\label{3}\\
&(L^{\K}F)(\gamma)=(L^{\mathrm{int}}F)(\gamma)+(L^{\mathrm{ext}}F)(\gamma),\label{4}
\end{align}
with
\begin{align*}(\Delta^X_xF)(\gamma):=&\Delta_y\big|_{y=x}F(\gamma-\delta_{(x,s)}+\delta_{(y,s)}),\\
 (\Delta_{x}^\Rp F)(\gamma):=&\bigg(\frac{d^2}{du^2}-\frac{d}{du}\bigg)\Big|_{u=s}F(\gamma-\delta_{(x,s)}+\delta_{(x,u)}).
\end{align*}

We equivalently have to prove that the symmetric operator
$(L^\sharp,\linebreak \FCC)$  is essentially self-adjoint on $L^2(\Gamma(\hat X),\pi)$.
 Denote by\linebreak  $(H^\sharp, D(H^\sharp))$ the closure of this symmetric operator  on $L^2(\Gamma(\hat X),\pi)$.
 So we have to prove that the operator $(H^\sharp, D(H^\sharp))$ is self-adjoint.

 It is not hard to check by approximation that, for each $\varphi\in\mathcal D(\hat X)$ and $n\in\N$, $F=\la \varphi,\cdot\ra^n \in D(H^\sharp)$ and  $(H^\sharp F)(\gamma)$ is given by the right hand sides of formulas \eqref{2}--\eqref{4}, respectively. Hence, by the polarization identity (e.g.\ \cite[Chap.~2, formula (2.17)]{BK}), we have
 \begin{equation}\label{cdrte6}
 \la \varphi_1,\cdot\ra\dotsm\la \varphi_n,\cdot\ra\in D(H^\sharp),\quad \varphi_1,\dots,\varphi_n\in\mathcal D(\hat X),\ n\in\N,\end{equation}
 and again the action of $H^\sharp$ onto a  function $F$ as in \eqref{cdrte6} is given by  the right hand side of formulas \eqref{2}--\eqref{4}, respectively. Let $\mathcal P$ denote the set of all functions on $\Gamma(\hat X)$ which are finite sums of functions as in \eqref{cdrte6} and constants. Thus, $\mathcal P$ is a set of polynomials on $\Gamma(\hat X)$, and $\mathcal P\subset D(H^\sharp)$.
 Furthermore,
 \begin{equation}\label{rdrde6}
 (-H^\sharp F,G)_{L^2(\Gamma(\hat X),\pi)}=\mathcal E^\sharp(F,G),\quad F,G\in\mathcal P,\ \sharp=\mathrm{int}, \mathrm{ext},\M.
 \end{equation}
 In formula \eqref{rdrde6}, $\mathcal E^\sharp(F,G)$ is given by formulas \eqref{ufr8}.

For a real separable  Hilbert space $\mathcal H$, we denote
by $\mathcal F(\mathcal H)$ the symmetric Fock space over $\mathcal H$. Thus, $\mathcal F(\mathcal H)$ is the real Hilbert space
$$\mathcal F(\mathcal H)=\bigoplus_{n=0}^\infty \mathcal F^{(n)}(\mathcal H),$$
where $\mathcal F^{(0)}(\mathcal H):=\R$, and for $n\in\N$, $\mathcal F^{(n)}(\mathcal H)$ coincides with $\mathcal H^{\odot n}$ as a set, and for any $f^{(n)},g^{(n)}\in\mathcal F^{(n)}(\mathcal H)$
$$ (f^{(n)},g^{(n)})_{\mathcal F^{(n)}(\mathcal H)}:=(f^{(n)},g^{(n)})_{\mathcal H^{\odot n}}\,n!\,.$$
Here $\odot$ stands for symmetric tensor product.

Recall the measure $\varkappa$ on $\hat X$ defined by formulas \eqref{5}, \eqref{biulgtfi}. Let
\begin{equation}\label{tfdty}
I: L^2(\Gamma(\hat X),\pi)\to\mathcal F(L^2(\hat X,\varkappa))\end{equation}
denote the unitary isomorphism which is derived through multiple stochastic integrals with respect to the centered Poisson random measure on $\hat X$ with intensity measure $\varkappa$, see e.g.\ \cite{Surgailis}. Denote by $\tilde{\mathcal P}$ the subset of $\mathcal F(L^2(\hat X,\varkappa))$ which is the linear span of vectors of the form
$$ \varphi_1\odot \varphi_2\odot\dots\odot \varphi_n,\quad \varphi_1,\dots,\varphi_n\in\mathcal D(\hat X),\ n\in\N$$
and the vacuum vector $\Psi=(1,0,0,\dots)$. For any $\varphi\in\mathcal D(\hat X)$, denote by $M_{\varphi}$ the operator of multiplication  by the function $\la\varphi,\cdot\ra$ in $L^2(\Gamma(\hat X),\pi)$. Using the representation of the operator $IM_{\varphi}I^{-1}$ as a sum of creation, neutral, and annihilation operators in the Fock space (see e.g.\ \cite{Surgailis}), we easily conclude that $I\mathcal P=\tilde{\mathcal P}$.

We define a bilinear form $(\tilde {\mathcal E}{}^\sharp,\tilde{\mathcal P})$ by
$$ \tilde {\mathcal E}{}^\sharp (f,g):=\mathcal E^\sharp(I^{-1}f,I^{-1}g),\quad f,g\in \tilde{\mathcal P} $$
on $\mathcal F(L^2(\hat X,\varkappa))$.

For each $(x,s)\in\hat X$, we define an annihilation operator at $(x,s)$ as follows:
$$\partial_{(x,s)}:\tilde{\mathcal P}\to \tilde{\mathcal P}$$ is the linear map given by
\begin{equation}\label{serasrwe}\partial_{(x,s)}\Psi:=0,\quad \partial_{(x,s)}\varphi_1\odot\varphi_2\odot\dots\odot\varphi_n:=
\sum_{i=1}^n \varphi_i(x,s)\varphi_1\odot\varphi_2\odot\dots\odot\check\varphi_i\odot\dots\odot\varphi_n,\end{equation}
where $\check\varphi_i$ denotes the absence of $\varphi_i$. We will preserve the notation $\partial_{(x,s)}$ for the operator $I\partial_{(x,s)}I^{-1}:\mathcal P\to\mathcal P$. This operator admits the following explicit representation:
$$\partial_{(x,s)}F(\gamma)=F(\gamma+\delta_{(x,s)})-F(\gamma)$$
for $\pi$-a.a.\ $\gamma\in\Gamma(\hat X)$, see e.g.\
\cite{IK,NV}.
Note that
\begin{align*}\nabla_x F(\gamma+\delta_{(x,s)})&=\nabla_x\big( F(\gamma+\delta_{(x,s)})-F(\gamma)\big),\\
\frac{d}{ds}\, F(\gamma+\delta_{(x,s)})&=\frac{d}{ds}\big( F(\gamma+\delta_{(x,s)})-F(\gamma)\big).
\end{align*}
Hence, by \eqref{ufr8}, for any $F,G\in\mathcal P$,
\begin{align}
&\mathcal E^{\mathrm{int}}(F,G)=\int_{\Gamma(\hat X)}d\pi(\gamma)\int_{\hat X}dx\,ds\,e^{-s}\,\frac{c(s)}{s^2}
\big\la\nabla_x\,\partial_{(x,s)}F(\gamma)
,\nabla_x\,\partial_{(x,s)}
G(\gamma)\big\ra_X\,,\notag\\
&\mathcal E^{\mathrm{ext}}(F,G)=\int_{\Gamma(\hat X)}d\pi(\gamma)\int_{\hat X}dx\,ds\,e^{-s}\left(\frac{\partial}{\partial s}
\,\partial_{(x,s)}F(\gamma)\right)\left(\frac{\partial}{\partial s}\partial_{(x,s)}
G(\gamma)\right),\notag\\
&\mathcal E^{\M}(F,G)=\mathcal E^{\mathrm{int}}(F,G)+\mathcal E^{\mathrm{ext}}(F,G).\notag
\end{align}
Hence, for any $f,g\in\tilde {\mathcal P}$,
\begin{align}
&\tilde {\mathcal E}{}^{\mathrm{int}}(f,g)=
\int_{\hat X}d\varkappa(x,s)\,
\frac{c(s)}{s}
\sum_{i=1}^d
\left(
\frac{\partial}{\partial x^i}\,\partial_{(x,s)}
f, \frac{\partial}{\partial x^i}\,\partial_{(x,s)}
g
\right)_{\mathcal F(L^2(\hat X,\varkappa))},\notag\\
&\tilde{\mathcal E}{}^{\mathrm{ext}}(f,g)=
\int_{\hat X}d\varkappa(x,s)\,s
\left(
\frac{\partial}{\partial s}\,\partial_{(x,s)}
f, \frac{\partial}{\partial s}\,\partial_{(x,s)}
g
\right)_{\mathcal F(L^2(\hat X,\varkappa))},\notag\\
&\tilde{\mathcal E}{}^{\M}(f,g)=\tilde{\mathcal E}{}^{\mathrm{int}}(f,g)+\tilde{\mathcal E}{}^{\mathrm{ext}}(f,g).\label{fut7re75i}
\end{align}

Consider the bilinear forms
\begin{align}
&\mathfrak E^{\mathrm{int}}(\varphi,\psi):=\int_{\hat X}d\varkappa(x,s)\,\frac{c(s)}{s} \big\la \nabla_x \varphi(x,s),\nabla_x\psi(x,s)\big\ra_X,\notag\\
&\mathfrak E^{\mathrm{ext}}(\varphi,\psi):=\int_{\hat X}d\varkappa(x,s)\,s \left( \frac{\partial}{\partial s}\varphi(x,s)\right)\left(\frac{\partial}{\partial s}\psi(x,s)\right),\notag\\
&\mathfrak E^{\M}(\varphi,\psi):=\mathfrak E^{\mathrm{int}}(\varphi,\psi)+\mathfrak E^{\mathrm{ext}}(\varphi,\psi),\quad \varphi,\psi\in\mathcal D(\hat X),
\label{yu7e5i}
\end{align}
on $L^2(\hat X,\varkappa)$. We easily calculate  the $L^2$-generators of these bilinear forms:
\begin{equation}\label{fgdtrse6t}
\mathfrak E^\sharp(\varphi,\psi)=(-\mathfrak L^\sharp\varphi,\psi)_{L^2(\hat X,\varkappa)},\quad \varphi,\psi \in \mathcal D(\hat X),\end{equation}
where for $\varphi\in\mathcal D(\hat X)$
\begin{align}
&(\mathfrak L^{\mathrm{int}}\varphi)(x,s)=\frac{c(s)}{s}\,\Delta_x\varphi(x,s),\notag\\
&(\mathfrak L^{\mathrm{ext}}\varphi)(x,s)=s\left(
\frac{\partial^2}{\partial s^2}-\frac{\partial}{\partial s}
\right)\varphi(x,s),\notag\\
&\mathfrak L^{\M}\varphi= \mathfrak L^{\mathrm{int}}\varphi+ \mathfrak L^{\mathrm{ext}}\varphi=\frac{c(s)}{s}\,\Delta_x\varphi(x,s)+s\left(
\frac{\partial^2}{\partial s^2}-\frac{\partial}{\partial s}
\right)\varphi(x,s).
\label{ui8t58rf}
\end{align}

Let us now recall the notion of a differential second
quantization. Let $(\mathcal A,\mathcal D)$ be a densely defined
symmetric operator in a real, separable Hilbert space $\mathcal H$.
We denote by $\mathcal F_{\mathrm{alg}}(\mathcal D)$
the subset of the Fock space $\mathcal F(\mathcal H)$ which is the linear span of the vacuum vector $\Psi$ and vectors of the form
$\varphi_1\odot \varphi_2\odot\dots\odot \varphi_n$, where $\varphi_1,\dots,\varphi_n\in\mathcal D$ and $n\in\N$. The differential second quantization $d\operatorname{Exp}(\mathcal A)$ is defined as the symmetric operator in $\mathcal F(\mathcal H)$ with domain $\mathcal F_{\mathrm{alg}}(\mathcal D)$ which acts as follows:
\begin{gather}
d\operatorname{Exp}(\mathcal A)\Psi:=0,\notag\\
d\operatorname{Exp}(\mathcal A)\varphi_1\odot \varphi_2\odot\dots\odot \varphi_n:=\sum_{i=1}^n \varphi_1\odot \varphi_2\odot\dots\odot
(\mathcal A\varphi_i)\odot\dots\odot \varphi_n.\label{uyrde56e6}
\end{gather}
By e.g.\ \cite[Chap.~6, subsec. 1.1]{BK}, if the operator $(\mathcal A,\mathcal D)$ is essentially self-adjoint on $\mathcal H$, then the differential second quantization $(d\operatorname{Exp}(\mathcal A), \mathcal F_{\mathrm{alg}}(\mathcal D))$ is essentially self-adjoint on $\mathcal F(\mathcal H)$.

Now, we note that $\tilde{\mathcal P}=\mathcal F_{\mathrm{alg}}(\mathcal D(\hat X))$.
By \eqref{serasrwe}--\eqref{uyrde56e6} (see also \cite[Chap.~6, Sect.~1]{BK}), an easy calculation shows that
$$\tilde{\mathcal E}{}^\sharp(f,g)=(d\operatorname{Exp}(-\mathfrak L^\sharp)f,g)_{\mathcal F(L^2(\hat X,\varkappa))}\,,\quad
f,g\in\tilde{\mathcal P},\ \sharp=\mathrm{int}, \mathrm{ext},\M.$$
Hence, by \eqref{rdrde6},
\begin{equation}\label{yr6e6i}
\tilde H^\sharp f=d\operatorname{Exp}(\mathfrak L^\sharp)f,\quad f\in\tilde{\mathcal P},\   \sharp=\mathrm{int}, \mathrm{ext},\M.\end{equation}
Here $\tilde H^\sharp:= IH^\sharp I^{-1}$.
To prove the theorem, it suffices to show that the operator $(H^\sharp,\mathcal P)$ is essentially self-adjoint on $L^2(\K(X),\mathcal G)$, or equivalently the operator $(\tilde H^\sharp,\tilde{\mathcal P})$ is essentially self-adjoint on $\mathcal F(L^2(\hat X,\varkappa))$. By \eqref{yr6e6i}, the theorem will follow from the lemma below.
\end{proof}

\begin{lemma}\label{gftyrebuhj} Under the assumptions of Theorem~\ref{ydred6e6r}, the  operator $(\mathfrak L^\sharp,\mathcal D(\hat X))$ is essentially self-adjoint on $L^2(\hat X,\varkappa)$, $\sharp=\mathrm{int}, \mathrm{ext},\M$.
\end{lemma}

\begin{proof}
We will only discuss the hardest case $\sharp=\M$. We denote by $(\mathfrak L^{\M},D(\mathfrak L^{\M}))$ the closure of the symmetric operator  $(\mathfrak L^{\M},\mathcal D(\hat X))$ on $L^2(\hat X,\varkappa)$.
We denote by $\mathcal S(X)$ the Schwartz space of real-valued, rapidly decreasing functions on $X$ (see e.g.\ \cite[Sect.~V.3]{ReedSimon1}).

{\it Claim.} For each $f\in \mathcal S(X)$ and $k\in\N$, the function $\varphi(x,s)=f(x)s^k$ belongs to $D(\mathfrak L^{\M})$, and $\mathfrak L^{\M}\varphi$ is given by the right hand side of \eqref{ui8t58rf}.

Indeed, for any functions $f\in\mathcal D(X)$ and $g\in C_0^\infty(\Rp)$, we have $f(x)g(s)\in\mathcal D(\hat X)\subset D(\mathfrak L^\K)$. Hence, by approximation, we can easily conclude that, for any functions $f\in\mathcal S(X)$ and $g\in C_0^\infty(\Rp)$, we have $f(x)g(s)\in D(\mathfrak L^\K)$.

Fix any function $u\in C^\infty(\R)$ such that $\chi_{[1,\,\infty)}\le u\le \chi_{[1/2,\,\infty)}$. Let
$$C_4:=\max_{t\in[1/2,\,1]}\max\{\,|u'(t)|,|u''(t)|\,\}<\infty.$$
For $n\in\N$, let $u_n(t):=u(nt)$, $t\in\R$. Then
\begin{equation}\label{it8}\chi_{[1/n,\,\infty)}\le u_n\le \chi_{[1/(2n),\,\infty)}\end{equation}
and
\begin{equation}\label{tdry7er}
|u'_n(t)|\le C_4\, n\chi_{[1/(2n),\, 1/n]}(t),\quad |u_n''(t)|\le C_4\, n^2\chi_{[1/(2n),\, 1/n]}(t),\quad t\in\R,\ n\in\N.
\end{equation}
We also fix any function $v\in C^\infty(\R)$ such that $\chi_{(-\infty,\,1]}\le v\le \chi_{(-\infty,\, 2]}$. For $n\in\N$, set $v_n(t):=v(t-n)$, $t\in\R$. Hence
\begin{equation}\label{gufr87}\chi_{(-\infty,\,n+1]}\le v_n\le \chi_{(-\infty,\, n+2]},\end{equation}
and for some $C_5>0$
\begin{equation}\label{igtf8t}
\max\{\,|v'_n(t)|,|v_n''(t)|\,\}\le C_5\,\chi_{[n+1,\,n+2]},\quad  t\in\R,\ n\in\N.
\end{equation}
We fix any  $k\in\N$ and set
\begin{equation}\label{ufuf} g_n(s):=s^k u_n(s)v_n(s),\quad s\in\Rp,\ n\in\N.\end{equation}
Clearly, $g_n\in C_0^\infty(\Rp)$. We  fix $f\in\mathcal S(X)$ and set
\begin{equation}\label{gfdyr6ed6}\varphi_n(x,s):=f(x)g_n(s),\quad (x,s)\in\hat X,\ n\in\N.\end{equation}
Thus, $\varphi_n\in D(\mathfrak L^\K)$.
By the dominated convergence theorem,
\begin{equation}\label{ye6jghh}\varphi_n(x,s)\to \varphi(x,s):=f(x)s^k\quad \text{in $L^2(\hat X,\varkappa)$  as $n\to\infty$}.
\end{equation}
We fix any $\psi\in\mathcal D(\hat X)$. Then
\begin{equation}\label{yyyyy}
(-\mathfrak L^\M\varphi_n,\psi)_{L^2(\hat X,\varkappa)}=\mathfrak E^\M(\varphi_n,\psi),\quad n\in\N.\end{equation}
It is easy to see  that
\begin{equation}\label{dr6e}
\lim_{n\to\infty} \mathfrak E^\M(\varphi_n,\psi)=\mathfrak E^\M(\varphi,\psi).\end{equation}
In   \eqref{yyyyy} and \eqref{dr6e} , $\mathfrak E^\M(\cdot,\cdot)$ is given by the formulas in \eqref{yu7e5i}. Hence
\begin{equation}\label{buy7rrf}
\lim_{n\to\infty}(\mathfrak L^\M\varphi_n,\psi)_{L^2(\hat X,\varkappa)}=(\mathfrak L^\M\varphi,\psi)_{L^2(\hat X,\varkappa)}.
\end{equation}
We stress that, in  \eqref{buy7rrf}, the function  $\mathfrak L^\M\varphi\in L^2(\hat X,\varkappa)$ is given by formulas in   \eqref{ui8t58rf}, however we do not yet state that $\varphi\in D(\mathfrak L^\M)$.

By using \eqref{it8}--\eqref{gfdyr6ed6}, it can be easily shown that
$$
\sup_{n\in\N}\|\mathfrak L^\M\varphi_n\|_{L^2(\hat X,\varkappa)}<\infty.
$$
Hence,
by the Banach--Alaoglu and Banach--Saks theorems (see e.g.\ \cite[Appendix, Sect.~2]{MR}), there exists a subsequence
$(\varphi_{n_j})_{j=1}^\infty$ of $(\varphi_n)_{n=1}^\infty$ such that the sequence $(\mathfrak L^\M\xi_i)_{i=1}^\infty$ converges in $L^2(\hat X,\varkappa)$. Here
$$\xi_i:=\frac1i\sum_{j=1}^i\varphi_{n_j},\quad i\in\N.$$
 We note that, for each $i\in\N$, $\xi_i\in\mathcal D(\hat X)$, and by \eqref{ye6jghh}
\begin{equation}\label{ufr7urf} \xi_i\to \varphi\quad \text{in $L^2(\hat X,\varkappa)$  as $i\to\infty$}.\end{equation}
Furthermore, by \eqref{buy7rrf},
$$ \lim_{i\to\infty}(\mathfrak L^\M\xi_i,\psi)_{L^2(\hat X,\varkappa)}=(\mathfrak L^\M\varphi,\psi)_{L^2(\hat X,\varkappa)},\quad \psi\in\mathcal D(\hat X).$$
Hence
\begin{equation}\label{ye6e}
\mathfrak L^\M\xi_i\to \mathfrak L^\M\varphi\quad \text{in $L^2(\hat X,\varkappa)$ as $i\to\infty$}.\end{equation}
By \eqref{ufr7urf} and \eqref{ye6e}, we conclude that $\xi_i\to\varphi$ in the graph norm of the operator $(\mathfrak L^\M,D(\mathfrak L^\M))$. Thus, the claim is proven.

We next note that
\begin{equation}\label{ufru7debjugh}L^2(\hat X,\varkappa)=L^2(X,dx)\otimes L^2(\Rp,\lambda)\end{equation}
 (recall \eqref{5}). Evidently, $\mathcal S(X)$ is a dense subset of $L^2(X,dx)$. Furthermore, the functions $\{s^k\}_{k=1}^\infty$ form a total set in $L^2(\Rp,\lambda)$ (i.e., the linear span of this set is dense in $L^2(\Rp,\lambda)$). Indeed, consider the unitary operator
$$ L^2(\Rp,\lambda)\ni g(s)\mapsto \frac{g(s)}s\,\in L^2(\Rp, se^{-s}\,ds).$$
Under this unitary operator, the set $\{s^k\}_{k=1}^\infty$ goes over into the set $\{s^k\}_{k=0}^\infty$. But the measure $\chi_{\Rp}(s)se^{-s}\,ds$ on $(\R,\mathcal B(\R))$ has Laplace transform which is analytic in a neighborhood of zero, hence the set of polynomials is dense in $ L^2(\Rp, se^{-s}\,ds)$. Therefore, the set
$$
\Upsilon:=
\operatorname{l.s.}\{f(x)s^k\mid f\in\mathcal S(X),\, k\in\N\}$$
is dense in $L^2(\hat X,\varkappa)$. Here $\operatorname{l.s.}$ denotes the linear span.
By the Claim, the set $\Upsilon$ is a subset of $D(\mathfrak L^\M)$. Note also that the operator $\mathfrak L^\M$ maps the set $\Upsilon$ into itself.

Since the symmetric operator $(\mathfrak L^\M,D(\mathfrak L^\M))$ is an extension of the operator $(\mathfrak L^\M,\Upsilon)$, to prove that $(\mathfrak L^\M,D(\mathfrak L^\M))$  is a self-adjoint operator, it suffices to prove that the operator  $(\mathfrak L^\M,\Upsilon)$ is essentially self-adjoint.

We denote by $L_\C^2(\hat X,\varkappa)$ the complex Hilbert space of all complex-valued $\varkappa$-square-integrable functions on $\hat X$.
Let  $\Upsilon_\C$ denote the complexification of $\Upsilon$, i.e., the set of all functions of the form $\varphi_1+i\varphi_2$, where $\varphi_1,\varphi_2\in \Upsilon$. Analogously, we define $L^2_\C(X,dx)$ and $\mathcal S_\C(X)$, the Schwartz space of complex-valued, rapidly decreasing functions on $X$.
 We extend the operator $\mathfrak L^{\M}$ by linearity to  $\Upsilon_\C$.

Recall that the Fourier transform determines a unitary operator $$\mathfrak F:L^2_\C(X,dx)\to L_\C^2(X,dx).$$ This operator leaves the Schwartz space $\mathcal S_\C(X)$ invariant, and furthermore
$$\mathfrak F:\mathcal S_\C(X)\to \mathcal S_\C(X)$$
is a bijective mapping. Under $\mathfrak F$, the Laplace operator $\Delta$ goes over into the operator of multiplication by $-\|x\|_X^2$, see e.g.\ \cite[Sect.~IX.1]{ReedSimon}.
Using \eqref{ufru7debjugh}, we obtain the unitary operator
$$\mathfrak F\otimes\mathbf 1: L_\C^2(\hat X,\varkappa)\to L_\C^2(\hat X,\varkappa).$$
Here $\mathbf 1$ denotes the identity operator.
Clearly
$\mathfrak F\otimes\mathbf 1:\Upsilon_\C\to\Upsilon_\C$
is a bijective mapping. We define an operator
$\mathfrak R^\M:\Upsilon_\C\to\Upsilon_\C$ by
$$\mathfrak R^\M:=(\mathfrak F\otimes\mathbf 1)\mathfrak L^\M(\mathfrak F\otimes\mathbf 1)^{-1}. $$
Explicitly, for each $\varphi\in\Upsilon_\C$,
\begin{equation}\label{utdei67ebhgl}
(\mathfrak R^\M\varphi)(x,s)=-\frac{c(s)}{s}\,\|x\|_{X}^2\, \varphi(x,s)+s\left(\frac{\partial^2}{\partial s^2}-\frac{\partial}{\partial s}\right)\varphi(x,s).\end{equation}
It suffices to prove that the operator $(\mathfrak R^\M,\Upsilon_\C)$ is essentially self-adjoint on $L_\C^2(\hat X,\varkappa)$.

Since the operator $(\mathfrak R^\M,\Upsilon_\C)$ is non-positive, by the Nussbaum theorem \cite{Nussbaum}, it suffices to prove that, for each function
\begin{equation}\label{tyde6i7e5}\varphi(x,s)=f(x)s^k\end{equation} with $f\in\mathcal D(X)$ and $k\in\N$,
\begin{equation}\label{gigtfi}
\sum_{n=1}^\infty \|(\mathfrak R^\M)^n \varphi\|_{L_\C^2(\hat X,\varkappa)}^{-1/2n}=\infty.\end{equation}
For a function $\varphi(x,s)$ of the form \eqref{tyde6i7e5}, by virtue of
\eqref{gyurt68otgyft} and \eqref{utdei67ebhgl}, we get
\begin{align}
&(\mathfrak R^\M\varphi)(x,s)\notag\\
&\quad=-(a_1 s^k+a_2s^{k+1}+a_3s^{k+2})\|x\|_X^2\,f(x)+(k(k-1)s^{k-1}-ks^k)f(x)\notag\\
&\quad=(\mathfrak R^\M_{-1}\varphi)(x,s)+(\mathfrak R^\M_{0}\varphi)(x,s)+(\mathfrak R^\M_{1}\varphi)(x,s)+(\mathfrak R^\M_{2}\varphi)(x,s).\label{jiyr}
\end{align}
Here
\begin{align}
(\mathfrak R^\M_{-1}\varphi)(x,s)&=k(k-1)s^{k-1}f(x),\notag\\
(\mathfrak R^\M_{0}\varphi)(x,s)&=(-a_1\|x\|_X^2-k)s^k f(x),\notag\\
(\mathfrak R^\M_{1}\varphi)(x,s)&=-a_2\|x\|_X^2s^{k+1} f(x),\notag\\
(\mathfrak R^\M_{2}\varphi)(x,s)&=-a_3\|x\|_X^2s^{k+2} f(x). \label{byur78f}
\end{align}
For $l\in\N$, denote
\begin{equation}\label{uytfr8u}
m_l:=\int_{\Rp}s^l\,d\lambda(s)=\int_0^\infty s^{l-1} e^{-s}\,ds.
\end{equation}
Since the Laplace transform of the measure $\chi_{\Rp}(s)e^{-s}\,ds$ on $\R$ is analytic in a neighborhood of zero, there exists a constant $C_6\ge1$ such that
\begin{equation}\label{tyei7} m_l\le C_6^l\, l!\,,\quad l\in\N.\end{equation}

Consider a product $\mathfrak R^\M_{i_1}\dotsm \mathfrak R^\M_{i_n}\varphi$, where $i_1,\dots,i_n\in\{-1,0,1,2\}$. Denote by
$l_j$ the number of the $\mathfrak R^\M_j$ operators among the operators $\mathfrak R^\M_{i_1},\dots, \mathfrak R^\M_{i_n}$.
Thus, $l_{-1}+l_0+l_1+l_2=n$.
Note that the function $f(x)$ has a compact support in $X$, hence the function $\|x\|_X^2$ is bounded on $\operatorname{supp}(f)$.
Recall also the estimate
\begin{equation}\label{t7re56e64}(2j)!\le 4^j\,(j!)^2,\quad j\in\N.\end{equation}
Hence, by \eqref{jiyr}--\eqref{t7re56e64}, we get:
\begin{equation}\label{igyi9gtf}
\|\mathfrak R^\M_{i_1}\dotsm \mathfrak R^\M_{i_n}\varphi\|_{L^2(\hat X,\varkappa)}\le C_{7}^n
(k-l_{-1}+l_1+2l_2)!\, (k-l_{-1}+l_1+2l_2)^{2l_{-1}+l_0}
\end{equation}
for some constant $C_{7}>0$ which is independent of $l_{-1},l_0,l_1,l_2,n$. Since $j!\le j^j$, we get from \eqref{igyi9gtf}
\begin{align}
\|\mathfrak R^\M_{i_1}\dotsm \mathfrak R^\M_{i_n}\varphi\|_{L^2(\hat X,\varkappa)}&\le C_{7}^n
(k-l_{-1}+l_1+2l_2)^{k-l_{-1}+l_1+2l_2+2l_{-1}+l_0}\notag\\
&=C_{7}^n(k-l_{-1}+l_1+2l_2)^{k+l_{-1}+l_0+l_1+2l_{2}}\notag\\
&\le C_{7}^n(k+2n)^{k+2n}.\notag
\end{align}
Therefore,
\begin{equation*}
\|(\mathfrak R^\M)^n \varphi\|_{L_\C^2(\hat X,\varkappa)}\le (4C_{7})^n(k+2n)^{k+2n}.
\end{equation*}
From here \eqref{gigtfi} follows.
\end{proof}

Let us recall the notion of a second quantization in a symmetric Fock space. Let $\mathcal H$ be a real separable  Hilbert space, and let  $\mathcal F(\mathcal H)$ be the symmetric Fock space over $\mathcal H$. Let $B$ be a bounded linear operator in $\mathcal H$, and assume that the operator norm of $B$ is $\le1$. We define the second quantization of $B$ as a bounded linear operator $\operatorname{Exp}(B)$ in $\mathcal H$ which satisfies $\operatorname{Exp}(B)\Psi:=\Psi$ ($\Psi$ being the vacuum vector in $\mathcal F(\mathcal H)$) and for each $n\in\N$, the restriction of $\operatorname{Exp}(B)$ to $\mathcal F^{(n)}(\mathcal H)$ coincides with $B^{\otimes n}$.

Let the conditions of Theorem~\ref{ydred6e6r} be satisfied.
For  $\sharp=\mathrm{int},\mathrm{ext},\M$, recall the non-positive self-adjoint operator $(\mathfrak L^\sharp,D(\mathfrak L^\sharp))$ in  $L^2(\hat X,\varkappa)$. By Lemma~\ref{gftyrebuhj}, this operator is essentially self-adjoint on  $\mathcal D(\hat X)$ and, for each $\varphi\in \mathcal D(\hat X)$, $\mathfrak L^\sharp\varphi$ is given by  \eqref{ui8t58rf}.
Recall the unitary operator $I$ in formula \eqref{tfdty}. In view of the bijective mapping $\mathcal R  :\Gamma_{pf}(\hat X)\to\mathbb K(X) $, we can equivalently treat the operator $I$ as a unitary operator
\begin{equation}\label{yitu} I: L^2(\K( X),\mathcal G)\to\mathcal F(L^2(\hat X,\varkappa))\end{equation}
(recall that the Poisson measure $\pi$ is concentrated on $\Gamma_{pf}(\hat X)$).

 \begin{corollary}\label{fufr}
Let the conditions of Theorem~\ref{ydred6e6r} be satisfied. Then, for $\sharp=\mathrm{int},\mathrm{ext},\M$,
we have
$$ Ie^{tL^\sharp}I^{-1}=\operatorname{Exp}(e^{t\mathfrak L^\sharp}),\quad t\ge0,$$
i.e., under the unitary isomorphism \eqref{yitu}, the semigroup $(e^{tL^\sharp})_{t\ge0}$
with generator $(L^\sharp,\linebreak[1] D(L^\sharp))$
goes over into the semigroup $(\operatorname{Exp}(e^{t\mathfrak L^\sharp}))_{t\ge0}$ --- the second
quantization of the semigroup $(e^{t\mathfrak L^\sharp})_{t\ge0}$ with generator $(\mathfrak L^\sharp,D(\mathfrak L^\sharp))$.
\end{corollary}

\begin{proof} It follows from the proof of Theorem~\ref{ydred6e6r} that
$$ I L^\sharp I^{-1}f=d\operatorname{Exp}(\mathfrak L^\sharp)f,\quad f\in\mathcal F_{\mathrm{alg}}(\mathcal D(\hat X)),$$
and the operator $d\operatorname{Exp}(\mathfrak L^\sharp)$ is essentially self-adjoint on $\mathcal F_{\mathrm{alg}}(\mathcal D(\hat X))$. From here the result immediately follows (cf.\ e.g.\ \cite[Chap.~6, subsec.~1.1]{BK}).
\end{proof}

\begin{remark}\label{opiugfyugf}
Consider the operator $(\mathfrak L^{\mathrm{ext}},D(\mathfrak L^{\mathrm{ext}}))$. We define  the linear operator
$$\mathfrak L^\mathrm{ext}_\Rp u(s):=s\left(
\frac{\partial^2}{\partial s^2}-\frac{\partial}{\partial s}
\right) u(s),\quad u\in C_0^\infty(\Rp).$$
It follows from the proof of Lemma~\ref{gftyrebuhj} that this operator is essentially self-adjoint on $L^2(\Rp,\lambda)$, and we denote by $(\mathfrak L^\mathrm{ext}_\Rp,D(\mathfrak L^\mathrm{ext}_\Rp))$ the closure of this operator.
 Recall that $L^2(\hat X,\varkappa)=L^2(X,dx)\otimes L^2(\Rp,\lambda)$. Using \eqref{ui8t58rf}, it is easy to show that
$$ \mathfrak L^\mathrm{ext}=\mathbf1\otimes \mathfrak L^\mathrm{ext}_\Rp .$$
Using e.g.\ \cite[Chap.~XI]{RY},
we easily conclude that $(\mathfrak L^\mathrm{ext}_\Rp,D(\mathfrak L^\mathrm{ext}_\Rp))$ is the generator of the Markov process $Y(t)$ on $\R_+=[0,\infty)$ given by the following space-time transformation of the square of the $0$-dimensional Bessel process $Q(t)$:
$$Y(t)=e^{-2t}Q((e^{2t}-1)/2).$$
Note that, for each  starting point $s>0$, the process $Y(t)$ is at 0 (so that it has exited  $\Rp$) with probability $\exp(-s/(1-e^{-t}))$, and once $Y(t)$ reaches zero it stays there forever (i.e., does not return to $\Rp$).
\end{remark}

\subsection*{Acknowledgment}
The authors   acknowledge the financial support of the SFB~701 ``Spectral structures and topological methods in mathematics'' (Bielefeld University) and the Research Group
``Stochastic Dynamics: Mathematical Theory and Applications'' (Center for Interdisciplinary Research, Bielefeld University).


\begin{thebibliography}{99}

\bibitem{AKR} S. Albeverio, Yu.G. Kondratiev, M.  R\"ockner,
 \textit{Analysis and geometry on configuration spaces}. J. Funct. Anal. {\bf 154} (1998),  444--500.


\bibitem{AKR2} S. Albeverio, Yu.G. Kondratiev, M.  R\"ockner, {\it Analysis and geometry on configuration spaces: the Gibbsian case}. J. Funct. Anal. {\bf 157} (1998), 242--291.





\bibitem{BK} Y.M. Berezansky, Y.G. Kondratiev, {\it Spectral methods in infinite-dimensional analysis. Vol. 1, 2}.  Kluwer Academic Publishers, Dordrecht, 1995.

\bibitem{B}  V.I. Bogachev, {\it Gaussian measures}.  American Mathematical Society, Providence, RI, 1998.




 \bibitem{VGG1}    I.M Gel'fand, M.I. Graev, A.M. Vershik, {\it Models of representations of current groups}.   Representations of Lie groups and Lie algebras (Budapest, 1971), 121--179, Akad. Kiad\'o, Budapest, 1985.

\bibitem{Hagedorn} D. Hagedorn, D., {\it Stochastic analysis related to gamma measures --- Gibbs perturbations and associated diffusions}, PhD Thesis, University of Bielefeld, 2011, available at http://pub.uni-bielefeld.de/publication/2452008

\bibitem{HKL} Y.  Kondratiev, E.  Lytvynov, {\it Equilibrium dynamics on the cone of discrete Radon measures}, in preparation.

 \bibitem{HKPR}  D. Hagedorn, Y. Kondratiev, T. Pasurek, M. R\"ockner,  {\it Gibbs states over the cone of discrete measures}. J. Funct. Anal. {\bf 264} (2013),  2550--2583.

\bibitem{HKPS}  T. Hida, H.H. Kuo, J. Potthoff, L. Streit,   {\it White noise. An infinite-dimensional calculus.}  Kluwer Academic Publishers Group, Dordrecht, 1993.


\bibitem{IK} Y. Ito, I. Kubo, 
{\it Calculus on Gaussian and Poisson white noises}.
Nagoya Math. J. {\bf 111} (1988), 41--84.


    \bibitem{Kal} O. Kallenberg, {\it Random measures}. Akad.-Verl., Berlin, 1983


\bibitem{KL}  Y.G. Kondratiev, E.W. Lytvynov,  {\it Operators of gamma white noise calculus}. Infin. Dimens. Anal. Quantum Probab. Relat. Top. {\bf 3} (2000),  303--335.

\bibitem{KLV} Y.G. Kondratiev, E. Lytvynov, A.M. Vershik,  {\it Laplace operators on the cone of Radon measures}, in preparation.

 \bibitem{KSSU} Y.G. Kondratiev, J.L. da Silva, L.  Streit, G.  Us, {\it Analysis on Poisson and gamma spaces}. Infin. Dimens. Anal. Quantum Probab. Relat. Top. {\bf 1} (1998), 91--117.

\bibitem{LS}  M.A. Lifshits,  E.Yu.\ Shmileva,  {\it Poisson measures that are quasi-invariant with respect to multiplicative transformations}.  Theory Probab. Appl. 46 (2003),  652--666.

\bibitem{L1}  E. Lytvynov,  {\it Polynomials of Meixner's type in infinite dimensions---Jacobi fields and orthogonality measures}. J. Funct. Anal. {\bf 200} (2003), 118--149.

\bibitem{L2} E. Lytvynov, {\it Orthogonal decompositions for L\'evy processes with an application to the gamma, Pascal, and Meixner processes}. Infin. Dimens. Anal. Quantum Probab. Relat. Top. {\bf 6} (2003),  73--102.

\bibitem{MR} Z.M. Ma, M. R\"ockner, M.,
{\it Introduction to the theory of (nonsymmetric) Dirichlet forms}.
Universitext, Springer-Verlag, Berlin, 1992

\bibitem{MRpaper} Z.-M. Ma, M.  R\"ockner,
{\it Construction of diffusions on configuration spaces}.
Osaka J. Math. {\bf 37} (2000),  273--314


\bibitem{NV} D. Nualart, J.   Vives,  {\it Anticipative calculus for the Poisson process based on the Fock space}. S\'eminaire de Probabilit\'es, XXIV, 1988/89, 154--165, Lecture Notes in Math., 1426, Springer, Berlin, 1990.

\bibitem{Nussbaum}  A.E. Nussbaum,  {\it A note on quasi-analytic vectors}. Studia Math. {\bf 33} (1969), 305--309.



\bibitem{osada}  H. Osada, {\it Dirichlet form approach to infinite-dimensional Wiener process with
singular interactions}. Comm. Math. Phys. {\bf 176} (1996), 117--131.

    

\bibitem{ReedSimon1}  M. Reed, B. Simon, {\it Methods of modern mathematical physics. I. Functional analysis}. Academic Press, New York--London, 1972.

\bibitem{ReedSimon}  M. Reed, B.   Simon, {\it Methods of modern mathematical physics. II. Fourier analysis, self-adjointness}. Academic Press, New York--London, 1975.

\bibitem{RY} D. Revuz, M. Yor,   {\it Continuous martingales and Brownian motion}.  Springer-Verlag, Berlin, 1991.


\bibitem{Sk} A.V. Skorohod,  {\it On the differentiability of measures which correspond to stochastic
processes I. Processes with independent increments}, Teoria Verojat. Primen. {\bf 2}
(1957), 418--444.

\bibitem{Surgailis} D. Surgailis, {\it On multiple Poisson stochastic integrals and associated Markov semigroups}. Probab. Math. Statist. {\bf 3} (1984),  217--239.


\bibitem{TsVY} N. Tsilevich, A. Vershik, M. Yor, {\it An infinite-dimensional analogue of the Lebesgue measure and distinguished properties of the gamma process}. J. Funct. Anal. {\bf 185} (2001), 274--296.

\bibitem{V} A.M. Vershik, {\it Does a Lebesgue measure in an infinite-dimensional space exist?}  Proc. Steklov Inst. Math. {\bf 259} (2007), 248--272.


\bibitem{VGG3} A.M. Vershik, I.M. Gel'fand, M.I. Graev,  {\it Representation of $SL(2, R)$, where $R$ is a
ring of functions}. Russian Math. Surveys {\bf 28} (1973), 83--128. [English translation in
``Representation Theory,'' London Math. Soc. Lecture Note Ser., Vol. 69, pp. 15--60,
Cambridge Univ. Press, Cambridge, UK, 1982.]

\bibitem{VGG2}   A.M. Vershik, I.M. Gel'fand, M.I. Graev, {\it Commutative model of the representation of the group of flows $SL(2,\mathbf R)^X$ connected with a unipotent subgroup}.  Funct. Anal. Appl. {\bf 17} (1983), 80--82.

\bibitem{yoshida}  M.W. Yoshida,  {\it Construction of infinite-dimensional interacting diffusion processes
through Dirichlet forms. Probab}. Th. Rel. Fields {\bf 106} (1996), 265--297.


\end{thebibliography}
\end{document}